\documentclass[12pt,a4paper,reqno]{amsart}
\usepackage{amssymb}
\usepackage{amscd}
\usepackage{enumerate}
\usepackage{graphicx}
\usepackage{siunitx}
\usepackage{tikz-cd}
\usepackage{color}
\usetikzlibrary{arrows}
\numberwithin{equation}{section}

\usepackage{mathabx}

\usepackage{mathtools}%                  http://www.ctan.org/pkg/mathtools
\usepackage[tableposition=top]{caption}% http://www.ctan.org/pkg/caption
\usepackage{booktabs,dcolumn}%           http://www.ctan.org/pkg/dcolumn + http://www.ctan.org/pkg/booktabs

\DeclareFontFamily{OT1}{rsfs}{}
\DeclareFontShape{OT1}{rsfs}{n}{it}{<-> rsfs10}{}
\DeclareMathAlphabet{\mathscr}{OT1}{rsfs}{n}{it}

\theoremstyle{plain}

\newtheorem{theorem}{Theorem}[section]
\newtheorem{proposition}[theorem]{Proposition}
\newtheorem{lemma}[theorem]{Lemma}
\newtheorem{corollary}[theorem]{Corollary}
\newtheorem{conjecture}[theorem]{Conjecture}

\theoremstyle{definition}

\newtheorem{definition}[theorem]{Definition}
\newtheorem{remark}[theorem]{Remark}

\newtheorem{example}[theorem]{Example}

\newcommand\R{\mathbb{R}}
\newcommand\Z{\mathbb{Z}}

\newcommand\eps{\varepsilon}

\usepackage{algorithm, algorithmic}

\begin{document}

\title[Integration approach to square peg problem]{An integration approach to the Toeplitz square peg problem}

\author{Terence Tao}
\address{UCLA Department of Mathematics, Los Angeles, CA 90095-1555.}
\email{tao@math.ucla.edu}

%\email{}


\begin{abstract}  The ``square peg problem'' or ``inscribed square problem'' of Toeplitz asks if every simple closed curve in the plane inscribes a (non-degenerate) square, in the sense that all four vertices of that square lie on the curve.  By a variety of arguments of a ``homological'' nature, it is known that the answer to this question is positive if the curve is sufficiently regular.  The regularity hypotheses are needed to rule out the possibility of arbitrarily small squares that are inscribed or almost inscribed on the curve; because of this, these arguments do not appear to be robust enough to handle arbitrarily rough curves.

In this paper we augment the homological approach by introducing certain integrals associated to the curve.  This approach is able to give positive answers to the square peg problem in some new cases, for instance if the curve is the union of two Lipschitz graphs $f, g \colon [t_0,t_1] \to \R$ that agree at the endpoints, and whose Lipschitz constants are strictly less than one.  We also present some simpler variants of the square problem which seem particularly amenable to this integration approach, including a periodic version of the problem that is not subject to the problem of arbitrarily small squares (and remains open even for regular curves), as well as an almost purely combinatorial conjecture regarding the sign patterns of sums $y_1+y_2+y_3$ for $y_1,y_2,y_3$ ranging in finite sets of real numbers.
\end{abstract}

\maketitle

%%%%%%%%%%%%%%%%%%%%%%%%%%%%%%%%%%%%%%%%%%%%%%%%%

\section{Introduction}

A subset $\Gamma$ of the plane $\R^2$ is said to \emph{inscribe a square} if it contains the four vertices of a square of positive sidelength. 
Note that despite the terminology, we do not require the solid square with these four vertices to lie in the region enclosed by $\Gamma$ (in the case that $\Gamma$ is a closed curve); see Figure \ref{inscribed}.  

\begin{figure} [t]
\centering
\includegraphics[width=5in]{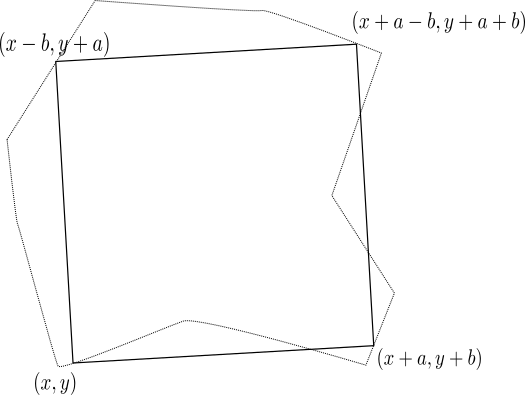}
\caption{A (simple, closed, polygonal) curve inscribing a square.}
\label{inscribed}
\end{figure}

The following conjecture of Toeplitz \cite{toeplitz} is usually referred to as the \emph{Square Peg Problem} or \emph{Inscribed Square Problem}:

\begin{conjecture}[Square Peg Problem]\label{spp}\cite{toeplitz}  Let $\gamma \colon \R/L\Z \to \R^2$ be a simple closed curve.  Then $\gamma(\R/L\Z)$ inscribes a square.
\end{conjecture}

In this paper, a curve $\gamma \colon I \to M$ denotes a continuous map from a domain $I$ that is either an interval $[t_0,t_1]$ or a circle $\R/L\Z$ to a manifold $M$, with the curve being called \emph{closed} in the latter case.  Such a curve is called \emph{simple} if $\gamma$ is injective.

A recent survey on this problem, together with an extensive list of further references may be found in \cite{mat}; the brief summary below is drawn from that survey.  

Thanks to arguments of a homological nature, the above conjecture is known assuming that the curve $\gamma$ is sufficiently regular.  For instance:
\begin{itemize}
\item Toeplitz \cite{toeplitz} claimed Conjecture \ref{spp} for convex curves, but did not publish a proof.  This case was partially resolved by Emch \cite{emch-1}, \cite{emch-2} and then fully resolved in \cite{zin}, \cite{christ}; this case can also be deduced from the ``table theorem'' of Fenn \cite{fenn}.
\item Hebbert \cite{hebbert} gave a proof of Conjecture \ref{spp} for quadrilaterals.  A proof of Conjecture \ref{spp} for arbitrary polygons was given in \cite{pak}; see also \cite{sag}, \cite{sag-2} for some further ``discretisations'' of the conjecture.  See \cite{pto} for some computer-assisted quantitative bounds on one such discretisation. 
\item Emch \cite{emch-3} established Conjecture \ref{spp} for piecewise analytic curves.  An alternate proof of the analytic case was obtained by Jerrard \cite{jerrard}.
\item Schnirelman \cite{sch} (see also \cite{gug}) established Conjecture \ref{spp} for continuously twice differentiable curves (and in fact was also able to treat some curves outside this class).  An alternate argument treating continuously twice differentiable curves (obeying an additional technical geometric condition) was obtained by Makeev \cite{makeev}.
\item Nielsen and Wright \cite{niel} established Conjecture \ref{spp} for curves symmetric around a line or point.
\item Stromquist \cite{strom} established Conjecture \ref{spp} for locally monotone curves.  An alternate treatment of this case was given in \cite{vz}.
\item Cantarella, Denne, McCleary \cite{cdm} established Conjecture \ref{spp} for $C^1$ curves, and also for bounded curvature curves without cusps.
\item Matschke \cite{matschke} established Conjecture \ref{spp} for an open dense class of curves (namely, curves that did not contain small trapezoids of a certain form), as well as curves that were contained in annuli in which the ratio between the outer and inner radius is at most $1+\sqrt{2}$.
\end{itemize} 

One can broadly classify the methods of proof in the above positive results as being ``homological'' in nature, in that they use algebraic topology tools such as the intersection product in homology, bordism, equivariant obstruction theory, or more elementary parity arguments counting the number of intersections between curves.  In fact, many of these proofs establish the stronger claim that (under some suitable genericity and regularity hypotheses) there are an \emph{odd} number of squares with vertices lying on the curve.

It is tempting to attack the general case of Conjecture \ref{spp} by approximating the curve $\gamma$ by a curve in one of the above classes, applying an existing result to produce squares with vertices on the approximating curve, and then taking limits somehow (e.g. by a compactness argument).  However, in the absence of some additional regularity hypothesis on the curve, it could conceivably happen that the approximating inscribed squares degenerate to a point in the limit.  Even if the original curve is smooth except for a single singular point, one could imagine that all squares in the approximating curves concentrate into that singular point in the limit.  This scenario of squares degenerating to a point is the primary reason for the inability to remove the regularity hypotheses in the known positive results on the problem. 

In this paper we propose to modify the homological approach to the inscribed square problem, by focusing not only on the parity of intersections between various geometric objects associated to the curve $\gamma$, but also on bounding certain integrals associated to these curves.  As with previous works, one requires a certain amount of regularity (such as rectifiability, Lipschitz continuity, or piecewise linearity) on the curves in order to initially define these integrals; but the integrals enjoy more stability properties under limits than intersection numbers, and thus may offer a route to establish more cases of Conjecture \ref{spp}.  As an instance of this, we give the following positive result towards this conjecture, which appears to be new.  For any $I \subset \R$ and any function $f \colon I \to \R$, we define the \emph{graph} $\mathtt{Graph}_f \colon I \to \R^2$ to be the function $\mathtt{Graph}_f(t) \coloneqq (t,f(t))$, so in particular $\mathtt{Graph}_f(I) \subset \R^2$ is the set
$$ \mathtt{Graph}_f(I) \coloneqq \{ (t, f(t)): t \in I \}.$$  
Such a function $f$ is said to be \emph{$C$-Lipschitz} for a given $C>0$ if one has $|f(s)-f(t)| \leq C|s-t|$ for all $s,t \in I$.  Similarly if $f$ is defined on a circle $\R/L\Z$ rather than an interval $I$ (using the usual Riemannian metric on $\R/L\Z$).

\begin{theorem}[The case of small Lipschitz constant]\label{main}  Let $[t_0,t_1]$ be an interval, and let $f,g \colon [t_0,t_1] \to \R$ be $(1-\eps)$-Lipschitz functions for some $\eps>0$.  Suppose also that $f(t_0)=g(t_0), f(t_1)=g(t_1)$, and $f(t) < g(t)$ for all $t_0 < t < t_1$.  Then the set
\begin{equation}\label{set} 
\mathtt{Graph}_f([t_0,t_1]) \cup \mathtt{Graph}_g([t_0,t_1])
\end{equation}
inscribes a square.
\end{theorem}

In other words, Conjecture \ref{spp} holds for curves that traverse two Lipschitz graphs, as long as the Lipschitz constants are strictly less than one.  The condition of having Lipschitz constant less than one is superficially similar to the property of being locally monotone, as considered in the references \cite{strom}, \cite{vz} above; however, due to a potentially unbounded amount of oscillation at the endpoints $\mathtt{Graph}_f(t_0) = \mathtt{Graph}_g(t_0)$ and $\mathtt{Graph}_f(t_1) = \mathtt{Graph}_g(t_1)$, the sets in Theorem \ref{main} are not necessarily locally monotone at the endpoints, and so the results in \cite{strom}, \cite{vz} do not directly imply Theorem \ref{main}.  Similarly for the other existing positive results on the square peg problem.

We prove Theorem \ref{main} in Section \ref{pos-sec}.  A key concept in the proof will be the notion of the (signed) area $\int_\gamma y\ dx$ under a rectifiable curve $\gamma$; see Definition \ref{area-def}.  This area can be used to construct a ``conserved integral of motion'' when one traverses the vertices of a continuous family of squares; see Lemma \ref{conserv}.  Theorem \ref{main} will then follow by applying the contraction mapping theorem to create such a continuous family of squares to which the conservation law can be applied, and then invoking the Jordan curve theorem to obtain a contradiction.  The hypothesis of $f,g$ having Lipschitz constant less than one is crucially used to ensure that the curve that one is applying the Jordan curve theorem to is simple; see Proposition \ref{pro}.  

Since the initial release of this paper, we have learned that a very similar method was introduced by Karasev \cite{karasev} to obtain a partial result on the related problem of Makeev \cite{makeev} of inscribing a given cyclic quadrilateral in a simple smooth closed curve, and indeed the proof of Theorem \ref{main} can be generalised to handle inscribing an equiangular trapezoid if the Lipschitz constants are small enough; see Remark \ref{bm}.  We thank Benjamin Matschke for this reference and observation.

Without the hypothesis of small Lipschitz constant, the Jordan curve theorem is no longer available, and we do not know how to adapt the argument to prove Conjecture \ref{spp} in full generality.  However, in later sections of the paper we present some related variants of Conjecture \ref{spp} which seem easier to attack by this method, including a periodic version (Conjecture \ref{pspp}), an area inequality version (Conjecture \ref{area-ineq}), and a combinatorial version (Conjecture \ref{adf}).  In contrast to the original square problem, the periodic version remains open even when the curves are piecewise linear (and this case seems to contain most of the essential difficulties of the problem).  Conjecture \ref{area-ineq} is the strongest of the new conjectures in this paper, implying all the other conjectures stated here except for the original Conjecture \ref{spp} (although it would imply some new cases of that conjecture).  Conjecture \ref{adf} is a simplified version of Conjecture \ref{area-ineq} that is an almost purely combinatorial claim regarding the sign patterns of triple sums $y_1+y_2+y_3$ of numbers $y_1, y_2, y_3$ drawn from finite sets of real numbers. It seems to be a particularly tractable ``toy model'' for the original problem, though the author was not able to fully resolve it.  The logical chain of dependencies between these conjectures, as well as some more technical variants of these conjectures that will be introduced in later sections, is summarised as follows, in which each conjecture is annotated with a very brief description:
$$
\left.
\begin{array}{ccccccc}
\stackrel{\text{(square peg)}}{\text{Conj. \ref{spp}}} & & \stackrel{\text{(quadripartite)}}{\text{Conj. \ref{qpspp}}} & \impliedby & \stackrel{\text{(area ineq.)}}{\text{Conj. \ref{area-ineq}}} & & \stackrel{\text{(combinatorial)}}{\text{Conj. \ref{adf}}} \\
\Downarrow & & \Downarrow & & \Downarrow & & \Updownarrow \\
\stackrel{\text{(special periodic)}}{\text{Conj. \ref{pspps}}} & \impliedby & \stackrel{\text{(periodic)}}{\text{Conj. \ref{pspp}}} & & \stackrel{\text{(special area ineq.)}}{\text{Conj. \ref{area-1d}}} & \iff & \stackrel{\text{(sym. special area ineq.)}}{\text{Conj. \ref{sai}}}
\end{array}
\right.
$$

The author is supported by NSF grant DMS-1266164, the James and Carol Collins Chair, and by a Simons Investigator Award.  He also thanks Mark Meyerson, Albert Hasse, the anonymous referees, and anonymous commenters on his blog for helpful comments, suggestions, and corrections.

\section{Notation}

Given a subset $E$ of a vector space $V$ and a shift $h \in V$, we define the translates
$$ E+h \coloneqq \{ p+h: p \in E \}.$$
Given two subsets $E,F$ in a metric space $(X,d)$, we define the distance
$$ \mathtt{dist}(E,F) \coloneqq \inf_{p \in E, q \in F} d(p,q).$$

We use the asymptotic notation $X=O(Y)$ to denote an estimate of the form $|X| \leq CY$ for some implied constant $C$; in many cases we will explicitly allow the implied constant $C$ to depend on some of the ambient parameters.  If $n$ is an asymptotic parameter going to infinity, we use $X=o(Y)$ to denote an estimate of the form $|X| \leq c(n) Y$ where $c(n)$ is a quantity depending on $n$ (and possibly on other ambient parameters) that goes to zero as $n \to \infty$ (holding all other parameters fixed).

We will use the language of singular homology (on smooth manifolds) in this paper, thus for instance a \emph{$1$-chain} in a manifold $M$ is a formal integer linear combination of curves $\gamma: I \to M$, and a \emph{$1$-cycle} is a $1$-chain $\sigma$ whose boundary $\partial \sigma$ vanishes.  Two $k$-cycles are \emph{homologous} if they differ by a $k$-boundary, that is to say the boundary $\partial U$ of a $k+1$-cycle.  We integrate (continuous) $k$-forms on (piecewise smooth) $k$-chains in the usual fashion, for instance if $\sigma = \sum_{i=1}^n c_i \gamma_i$ is a $1$-chain that is an integer linear combination of curves $\gamma_i$, and $\theta$ is a $1$-form, then $\int_\sigma \theta \coloneqq \sum_{i=1}^n c_i \int_{\gamma_i} \sigma$.
See for instance \cite{hatcher} for the formal definitions of these concepts and their basic properties.  We will not use particularly advanced facts from singular homology; perhaps the most important fact we will use is the claim that if two (piecewise linear) cycles $\gamma_1, \gamma_2$ in an oriented manifold are homologous and intersect a smooth oriented submanifold $V$ (without boundary) transversely, then their intersections $\gamma_1 \cap V$, $\gamma_2 \cap V$ are homologous cycles in $V$.  Indeed, if $\gamma_1$ and $\gamma_2$ differ by the boundary of some cycle $\sigma$, then $\gamma_1 \cap V$ and $\gamma_2 \cap V$ differ by the boundary of $\sigma \cap V$ (viewed as a cycle with an appropriate orientation); alternatively, one may use Poincar\'e duality and the theory of the cup product.

\section{Proof of positive result}\label{pos-sec}

We now prove Theorem \ref{main}.  
It will be convenient to give a name to the space of all squares.

\begin{definition}[Squares]\label{sqdef}
We define $\mathtt{Squares} \subset (\R^2)^4$ to be the set of all quadruples of vertices of squares in $\R^2$ traversed in anticlockwise order; more explicitly, we have
$$ \mathtt{Squares} \coloneqq \{ ((x,y), (x+a, y+b), (x+a-b, y+a+b), (x-b, y+a)): x,y,a,b \in \R; (a,b) \neq (0,0) \}.$$
By abuse of notation we refer to elements of $\mathtt{Squares}$ as (non-degenerate) squares.
Thus we see that a set $\Gamma \subset \R^2$ inscribes a square if and only if $\Gamma^4$ intersects $\mathtt{Squares}$.  We also form the closure 
$$ \overline{\mathtt{Squares}} \coloneqq \{ ((x,y), (x+a, y+b), (x+a-b, y+a+b), (x-b, y+a)): x,y,a,b \in \R \}$$
 in which the sidelength of the square is allowed to degenerate to zero; this is a four-dimensional linear subspace of $(\R^2)^4$.  A quadruple $(\gamma_1,\gamma_2,\gamma_3,\gamma_4)$ of curves $\gamma_1,\gamma_2,\gamma_3,\gamma_4 \colon [t_0,t_1] \to \R^2$ is said to \emph{traverse squares} if one has $(\gamma_1(t), \gamma_2(t), \gamma_3(t), \gamma_4(t)) \in \overline{\mathtt{Squares}}$ for all $t \in [t_0,t_1]$; note that we allow the square traversed to degenerate to a point.  Equivalently, $(\gamma_1,\gamma_2,\gamma_3,\gamma_4)$ traverses squares if and only if there exist continuous functions $x,y,a,b \colon [t_0,t_1] \to \R^2$ such that
\begin{align*}
\gamma_1(t) &= (x(t), y(t)) \\
\gamma_2(t) &= (x(t) + a(t), y(t) + b(t)) \\
\gamma_3(t) &= (x(t) + a(t) - b(t), y(t) + a(t) + b(t)) \\
\gamma_4(t) &= (x(t) - b(t), y(t) + a(t)).
\end{align*}
\end{definition}

We will also need the notion of the area under a (rectifiable) curve.  Recall that a curve $\gamma \colon [t_0,t_1] \to \R^2$ is rectifiable if the sums $\sum_{i=0}^{n-1} |\gamma(s_{i+1}) - \gamma(s_i)|$ are bounded for all partitions $t_0 = s_0 < s_1 < \dots < s_n = t_1$.  If we write $\gamma(t) = (x(t), y(t))$, this is equivalent to requiring that the functions $x, y \colon [t_0,t_1] \to \R^2$ are of bounded variation.

\begin{definition}[Area under a curve]\label{area-def}  Let $\gamma \colon [t_0,t_1] \to \R^2$ be a rectifiable curve, and write $\gamma(t) = (x(t),y(t))$ for $t \in [t_0,t_1]$.  The \emph{area under $\gamma$}, denoted by $\int_\gamma y\ dx$, is defined to be the real number
$$ \int_\gamma y\ dx \coloneqq \int_{t_0}^{t_1} y(t)\ dx(t)$$
where the integral on the right-hand side is in the Riemann-Stieltjes sense, that is to say the limit of $\sum_{i=0}^{n-1} y(s_i^*) (x(s_{i+1}) - x(s_i))$ for any partition $t_0 = s_0 < \dots < s_n = t_1$ and $s_i \leq s_i^* \leq s_{i+1}$ as $\max_{0 \leq i \leq n-1} |s_{i+1} - s_i|$ goes to zero.
\end{definition}

Of course, if $\gamma$ is piecewise smooth, this definition of $\int_\gamma y\ dx$ agrees with the usual definition of $\int_\gamma y\ dx$ as the integral of the $1$-form $y\ dx$ on $\gamma$ (now viewed as a $1$-chain).

\begin{example}\label{integ} If $f \colon [t_0,t_1] \to \R$ is continuous of bounded variation, then the area under the curve $\mathtt{Graph}_f$ is just
the usual Riemann integral:
$$ \int_{\mathtt{Graph}_f} y\ dx = \int_{t_0}^{t_1} f(t)\ dt.$$
In particular, if $f$ is positive, $\int_{\mathtt{Graph}_f} y\ dx$ is the area of the region bounded by $\mathtt{Graph}_f$, the real axis, and the vertical lines $\{t_0,t_1\} \times \R$.  If $\gamma$ is not a graph, the area under $\gamma$ is more complicated; see Figure \ref{area}.  
\end{example}

\begin{figure} [t]
\centering
\includegraphics[width=5in]{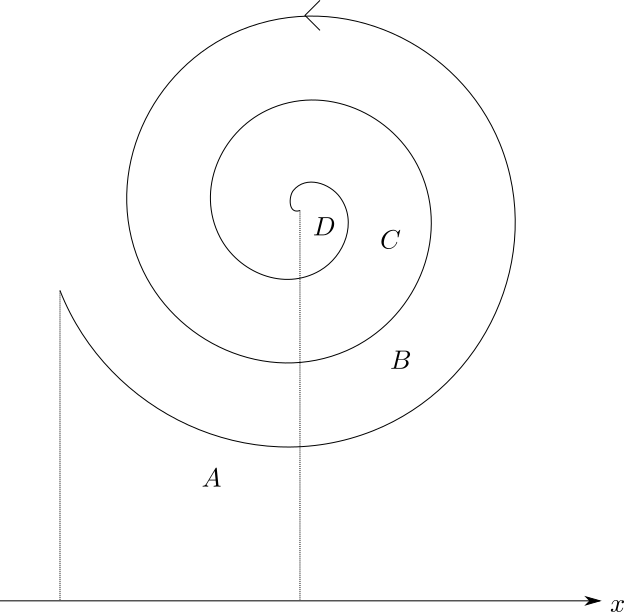}
\caption{The area under the spiral (with the anticlockwise orientation) is equal to $|A| - |B| - 2|C| - 3|D|$, where $|A|$ is the area of the region $A$, and similarly for $|B|$, $|C|$, $|D|$; the weights $+1, -1, -2, -3$ attached to $A,B,C,D$ here are the winding numbers of the spiral together with the horizontal and vertical line segments used to close up the curve.}
\label{area}
\end{figure}

It will be particularly important to understand the area under a closed curve:

\begin{lemma}\label{repo}  Let $\gamma \colon [t_0, t_1] \to \R^2$ be a simple closed anticlockwise rectifiable curve, and let $\Omega$ be the bounded open region enclosed by this curve as per the Jordan curve theorem.  Then the area under the curve $\gamma$ is then equal to the negative of the Lebesgue measure $|\Omega|$ of $\Omega$.  In particular, this area is non-zero.
\end{lemma}

Of course, if $\gamma$ were clockwise instead of anticlockwise, then the negative sign in the above lemma would be removed; however, it would still be true that the area under this curve is non-zero.

\begin{proof}
In the case that $\gamma$ is a polygonal path, this claim is clear from Stokes' theorem:
\begin{equation}\label{stok}
 \int_\gamma y\ dx = \int_{\partial \Omega} y\ dx = - \int_\Omega dx \wedge dy.
\end{equation}
Now we consider the general case.  The strategy is to approximate $\gamma$ by a polygonal path, apply \eqref{stok}, and take limits; but (as with the proof of the Jordan curve theorem) some care must be taken with the limiting argument.

We can normalise $[t_0,t_1] = [0,1]$.  Let $\eps_0>0$ be a small parameter (which will eventually be sent to zero).  By continuity of $\gamma$, there exists 
$0 < \eps_1 < \eps_0$ such that $|\gamma(t)-\gamma(t')| \leq \eps_0$ whenever $t \hbox{ mod } 1$ and $t' \hbox{ mod } 1$ are separated by distance at least $\eps_1$ on the unit circle $\R/\Z$.  By a further application of continuity and the hypothesis that $\gamma$ is simple, there exists $0 < \eps_2 < \eps_1$ such that $|\gamma(t)-\gamma(t')| \geq \eps_2$ whenever $t \hbox{ mod } 1$ and $t' \hbox{ mod } 1$ are separated by distance at least $\eps_1$ on the unit circle $\R/\Z$.  Now let $n$ be a natural number, that we assume to be sufficiently large depending on $\gamma, \eps_0, \eps_1, \eps_2$.  Let $\gamma_n: [0,1] \to \R^2$ be the piecewise polygonal path formed by joining up the points $\gamma_n(j/n) \coloneqq \gamma(j/n)$ for $j=0,\dots,n$ by line segments, thus 
$$\gamma_n\left(\frac{j+\theta}{n}\right) \coloneqq (1-\theta) \gamma\left(\frac{j}{n}\right) + \theta \gamma\left(\frac{j+1}{n}\right)$$
for $j=0,\dots,n-1$ and $0 \leq \theta \leq 1$.  As $\gamma$ is uniformly continuous, we see for $n$ large enough that
\begin{equation}\label{gaga}
 |\gamma_n(t) - \gamma(t)| < \frac{\eps_2}{2}
\end{equation}
for all $t \in [0,1]$.  Also, it is clear that the length of $\gamma_n$ is bounded by the length of the rectifiable curve $\gamma$.

The path $\gamma_n$ is closed, but it need not be simple.  However, from \eqref{gaga}, the triangle inequality, and the construction of $\eps_2$, we see that a collision $\gamma_n(t) = \gamma_n(t')$ can only occur if $t \hbox{ mod } 1$ and $t' \hbox{ mod } 1$ differ by at most $\eps_1$ in the unit circle.  In such a case, $\gamma_n$ can be viewed as the sum (in the sense of $1$-cycles) of two closed polygonal paths, one of which has diameter at most $\eps_0$.  Deleting the latter path and iterating, we conclude that $\gamma_n$ can be viewed as the sum of a simple closed polygonal path $\gamma_n^0$ and a finite number of closed polygonal paths $\gamma_n^1,\dots,\gamma_n^k$ of diameter at most $\eps_0$; furthermore, the total lengths of $\gamma_n^0, \gamma_n^1, \dots, \gamma_n^k$ sum up to at most the length of $\gamma$, and from \eqref{gaga} we see that the curves $\gamma_n^1,\dots,\gamma_n^k$ lie within the $2\eps_0$-neighbourhood of $\gamma$.

If $\eps_0$ is small enough, we can find a point $z$ in $\Omega$ that is at a distance at least $10 \eps_0$ from $\gamma$.  The winding number of $\gamma$ around $z$ is equal to $1$.  By Rouche's theorem, the winding number of $\gamma_n$ around $z$ is then also equal to $1$, while the winding numbers of $\gamma_n^1,\dots,\gamma_n^k$ around $z$ are equal to $0$.  Thus the winding number of $\gamma_n^0$ around $z$ is equal to $1$; thus $\gamma_n^0$ has an anticlockwise orientation, and $z$ lies in the region $\Omega_n$ enclosed by $\gamma_n^0$.  This argument also shows that the symmetric difference $\Omega \Delta \Omega_n$ between $\Omega$ and $\Omega_n$ lies in the $\eps_0$-neighbourhood of $\gamma$.  As $\gamma$ is rectifiable, this implies that
$$ | \Omega_n | = |\Omega| + O(\eps_0)$$
where the implied constant in the $O()$ notation depends on the length of $\gamma$.  On the other hand, from \eqref{stok} one has
$$ \int_{\gamma_n^0} y\ dx = - |\Omega_n|.$$
For each $i=1,\dots,k$, we see that
$$ \int_{\gamma_n^i} y\ dx = \int_{\gamma_n^i} (y - y^i)\ dx = O( \eps_0 \ell(\gamma_n^i) )$$
where $y_i$ is the $y$-coordinate of an arbitrary point in $\gamma_n^i$, and $\ell(\gamma_n^i)$ denotes the length of $\gamma_n^i$.  Summing, we conclude that
$$ \sum_{i=1}^k \int_{\gamma_n^i} y\ dx = O(\eps_0).$$
Finally, for $n$ sufficiently large, we have from the rectifiability of $\gamma$ that
$$ \int_{\gamma_n} y\ dx = \int_\gamma y\ dx + O(\eps_0).$$
Putting all these bounds together, we conclude that
$$\int_\gamma y\ dx = - |\Omega| + O(\eps_0);$$
since $\eps>0$ can be made arbitrarily small, the claim follows.
\end{proof}

The relevance of this area concept to the square peg problem lies in the following simple identity.

\begin{lemma}[Conserved integral of motion for squares]\label{conserv}  Let $\gamma_1, \gamma_2, \gamma_3, \gamma_4 \colon [t_0,t_1] \to \R^2$ be rectifiable curves such that $(\gamma_1,\gamma_2,\gamma_3,\gamma_4)$ traverses squares (as defined in Definition \ref{sqdef}).
Then we have the identity
\begin{equation}\label{ident}
\int_{\gamma_1} y\ dx - \int_{\gamma_2} y\ dx + \int_{\gamma_3} y\ dx - \int_{\gamma_4} y\ dx 
 = \frac{a(t_1)^2 - b(t_1)^2}{2} - \frac{a(t_0)^2 - b(t_0)^2}{2},
\end{equation}
where the functions $x,y,a,b \colon [t_0,t_1] \to \R$ are as in Definition \ref{sqdef}.
\end{lemma}

\begin{proof}  From Definition \ref{sqdef} and Definition \ref{area-def} we have
\begin{align*}
\int_{\gamma_1} y\ dx &= \int_{t_0}^{t_1} y(t)\ dx(t) \\
\int_{\gamma_2} y\ dx &= \int_{t_0}^{t_1} (y(t) + b(t))\ (dx(t) + da(t))\\
\int_{\gamma_3} y\ dx &= \int_{t_0}^{t_1} (y(t) + a(t) + b(t))\ (dx(t) + da(t) - db(t))\\
\int_{\gamma_4} y\ dx &= \int_{t_0}^{t_1} (y(t) + a(t))\ (dx(t) - db(t));
\end{align*}
note from the rectifiability of $\gamma_1, \gamma_2, \gamma_3, \gamma_4$ that $x,y,x+a,y+b$ (and hence $a,b$) are of bounded variation.
After some canceling, we may then write the left-hand side of \eqref{ident} as
$$
\int_{t_0}^{t_1} a(t)\ da(t) - \int_{t_0}^{t_1} b(t)\ db(t).
$$
Since $a,b$ are Lipschitz continuous, one has $a(s') (a(s')-a(s)) = \frac{1}{2} a(s')^2 - \frac{1}{2} a(s)^2 + O( |s-s'|^2 )$ for any $s,s' \in [t_0,t_1]$, which easily implies that
$$
\int_{t_0}^{t_1} a(t)\ da(t) = \frac{1}{2} a^2(t_1) - \frac{1}{2} a^2(t_0);$$ 
similarly we have
$$
\int_{t_0}^{t_1} b(t)\ db(t) = \frac{1}{2} b^2(t_1) - \frac{1}{2} b^2(t_0),$$ 
and the claim follows.
\end{proof}

\begin{remark} Geometrically, this conserved integral reflects the following elementary fact: if a square with vertices $p_1, p_2, p_3, p_4$ (traversed in anticlockwise order) and sidelength $l$ is deformed to a nearby square with vertices $p_1+dp_1$, $p_2+dp_2$, $p_3+dp_3$, $p_4+dp_4$ and sidelength $l+dl$, then the difference of the areas of the two quadrilaterals with vertices $p_1, p_1+dp_1, p_4+dp_4, p_4$ and $p_2,p_2+dp_2, p_3+dp_3, p_3$ respectively add up to exactly half the difference between the areas $l^2$, $(l+dl)^2$ of the two squares (see Figure \ref{diffi}).
\end{remark}

\begin{figure} [t]
\centering
\includegraphics[width=5in]{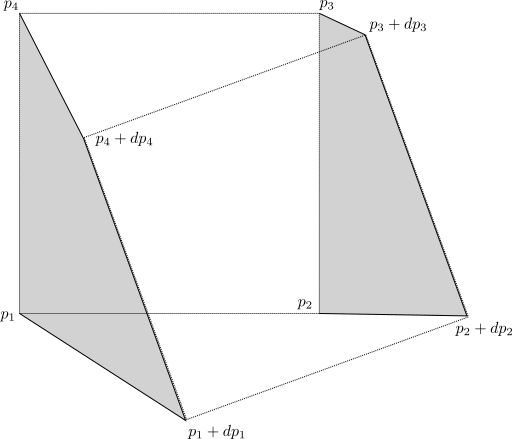}
\caption{The difference in areas between the shaded regions is half the difference in areas between the squares.  Taking ``Riemann sums'' of this fact will yield Lemma \ref{conserv}.}
\label{diffi}
\end{figure}

\begin{remark}\label{homog}  One can interpret Lemma \ref{conserv} in the language of differential forms as follows.  For $i=1,2,3,4$, let $\pi_i \colon \overline{\mathtt{Squares}} \to \R^2$ be the $i^{\operatorname{th}}$ coordinate projection, then we can pull back the $1$-form $y\ dx$ on $\R^2$ by $\pi_i$ to create a $1$-form $\pi_i^*(y\ dx)$ on the $4$-manifold $\overline{\mathtt{Squares}}$.  Then the identity \eqref{ident} may be rewritten as
$$ \pi_1^*(y\ dx) - \pi_2^*(y\ dx) + \pi_3^*(y\ dx) - \pi_4^*(y\ dx) = d\phi$$
where $\phi \colon \overline{\mathtt{Squares}} \to \R$ is the $0$-form that takes a square $((x,y), (x+a, y+b), (x+a-b, y+a+b), (x-b, y+a))$ to the quantity $\frac{a^2-b^2}{2}$, and $d$ denotes the exterior derivative.  In particular, the $1$-form $\pi_1^*(y\ dx) - \pi_2^*(y\ dx) + \pi_3^*(y\ dx) - \pi_4^*(y\ dx)$ is exact.
\end{remark}

Now we prove Theorem \ref{main}. Let $[t_0,t_1], f,g,\eps$ be as in that theorem.
It is convenient to extend the functions $f,g \colon [t_0,t_0] \to \R$ by constants to the whole real line $\R$ to form extended functions $\tilde f, \tilde g \colon \R \to \R$.  That is to say, we define $\tilde f(t)=\tilde g(t)\coloneqq f(t_1)=g(t_1)$ for $t>t_1$, $\tilde f(t)=\tilde g(t)\coloneqq f(t_0)=g(t_0)$ for all $t<t_0$, and $\tilde f(t) \coloneqq f(t)$ and $\tilde g(t) \coloneqq g(t)$ for $t_0 \leq t \leq t_1$.  Clearly $\tilde f,\tilde g \colon \R \to \R$ continue to be $(1-\eps)$-Lipschitz and of bounded variation.

For any $t \in \R$, the map
\begin{equation}\label{map}
 (a,b) \mapsto ( \tilde g(t-b) - \tilde f(t), \tilde f(t+a) - \tilde f(t) )
\end{equation}
is a strict contraction on $\R^2$ (with the usual Euclidean metric) with contraction constant at most $1-\eps$.  Hence, by the contraction mapping theorem (or Banach fixed point theorem) applied to the complete metric space $\R^2$, there is a unique solution $(a(t),b(t)) \in \R^2$ to the fixed point equation
\begin{equation}\label{fix}
 (a(t),b(t)) = ( \tilde g(t-b(t)) - \tilde f(t), \tilde f(t+a(t)) - \tilde f(t));
\end{equation}
furthermore, $a(t)$ and $b(t)$ depend in a Lipschitz fashion on $t$ (the Lipschitz constant can be as large as $O(1/\eps)$, but this will not concern us).  If we then define\footnote{Similar curves also appear in the arguments of Jerrard \cite{jerrard}.} the functions $\gamma_1,\gamma_2,\gamma_3, \gamma_4 \colon [t_0,t_1] \to \R^2$ by
\begin{align}
\gamma_1(t) &= (t, \tilde f(t)) \label{g1-def}\\
\gamma_2(t) &= (t+a(t), \tilde f(t) + b(t)) \label{g2-def}\\
\gamma_3(t) &= (t+a(t)-b(t), \tilde f(t) + a(t) + b(t)) \label{g3-def} \\
\gamma_4(t) &= (t-b(t), \tilde f(t) + a(t))\label{g4-def}
\end{align}
for all $t \in \R$, then $(\gamma_1,\gamma_2,\gamma_3,\gamma_4)$ is a quadruple of Lipschitz (and thus locally rectifiable) curves that traverse squares.  From \eqref{fix}, \eqref{g1-def}, \eqref{g2-def}, \eqref{g4-def} we have
\begin{align}
\gamma_1(t) &= \mathtt{Graph}_{\tilde f}(t) \label{g1-alt} \\
\gamma_2(t) &= \mathtt{Graph}_{\tilde f}(t+a(t)) \label{g2-alt}\\
\gamma_4(t) &= \mathtt{Graph}_{\tilde g}(t-b(t))\label{g4-alt}
\end{align}
for all $t \in \R$.  In particular, $\gamma_1, \gamma_2, \gamma_4$ take values in $\mathtt{Graph}_{\tilde f}(\R)$, $\mathtt{Graph}_{\tilde f}(\R)$, and $\mathtt{Graph}_{\tilde g}(\R)$ respectively; see Figure \ref{gammafig}.  As for $\gamma_3$, we can use the hypothesis of small Lipschitz constant to establish the following key fact:

\begin{figure} [t]
\centering
\includegraphics[width=5in]{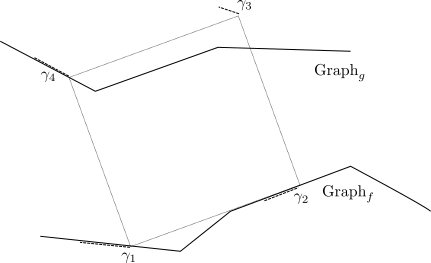}
\caption{Portions of the curves $\gamma_1, \gamma_2, \gamma_3, \gamma_4$, $\mathtt{Graph}_{f}$, and $\mathtt{Graph}_{g}$.  (In some cases, $\gamma_1,\gamma_2,\gamma_4$ may lie on the enlargements $\mathtt{Graph}_{\tilde f}, \mathtt{Graph}_{\tilde g}$ of $\mathtt{Graph}_f$, $\mathtt{Graph}_g$, which are not shown here.)}
\label{gammafig}
\end{figure}

\begin{proposition}\label{pro}  The curve $\gamma_3 \colon [t_0,t_1] \to \R^2$ is simple.
\end{proposition}

\begin{proof}  Suppose for contradiction that $t,t' \in [t_0,t_1]$ are distinct points such that $\gamma_3(t) = \gamma_3(t')$.  Then if one rotates the curve $\mathtt{Graph}_{\tilde f}$ clockwise by $\pi/2$ around $\gamma_3(t) = \gamma_3(t')$, the rotated curve will intersect $\mathtt{Graph}_{\tilde g}$ at the two distinct points $\gamma_2(t)$ and $\gamma_2(t')$ (see Figure \ref{gammafig}).  As $\tilde g$ is $1-\eps$-Lipschitz, we conclude that the line segment connecting these two points has slope of magnitude at most $1-\eps$ with respect to the $x$-axis; as $\tilde f$ is also $1-\eps$-Lipschitz, we similarly conclude that the same line segment has slope of magnitude at most $1-\eps$ with respect to the $y$-axis.  The two claims are inconsistent, giving the required contradiction.
\end{proof}

When $t=t_0$ or $t=t_1$, we have $(a(t),b(t)) = (0,0)$ as the unique fixed point of \eqref{fix}.  Applying Lemma \ref{conserv}, we conclude the identity
\begin{equation}\label{conc}
 \int_{\gamma_1} y\ dx - \int_{\gamma_2} y\ dx + \int_{\gamma_3} y\ dx - \int_{\gamma_4} y\ dx = 0.
\end{equation}
By \eqref{g1-alt} and Example \ref{integ} we have
\begin{equation}\label{g1a}
\int_{\gamma_1} y\ dx  = \int_{t_0}^{t_1} f(t)\ dt.
\end{equation}
From \eqref{g2-alt} and a change of variables\footnote{This change of variables is easy to justify if the map $t \mapsto t + a(t)$ is piecewise linear, and the general case follows by an approximation argument (noting that all functions involved are Lipschitz continuous).} $\tilde t \coloneqq t + a(t)$ we also have
\begin{equation}\label{g2a}
\int_{\gamma_2} y\ dx  = \int_{t_0}^{t_1} f(t)\ dt;
\end{equation}
note that $t+a(t)$ may be temporarily decreasing instead of increasing as $t$ increases from $t_0$ to $t_1$, but the net contributions of such excursions cancel out (by the fundamental theorem of calculus, or equivalently because $1$-forms on a line are automatically exact).  Similarly, from \eqref{g4-alt} and the change of variables $\tilde t \coloneqq t - b(t)$ we have
\begin{equation}\label{g4a}
\int_{\gamma_4} y\ dx = \int_{t_0}^{t_1} g(t)\ dt.
\end{equation}
Thus from \eqref{conc} we must have
$$\int_{\gamma_3} y\ dx = \int_{t_0}^{t_1} g(t)\ dt$$
or equivalently
$$\int_{\gamma_3 + (-\mathtt{Graph}_g)} y\ dx = 0$$
where $\gamma_3 + (-\mathtt{Graph}_g)$ denotes the concatenation of $\gamma_3$ with the reversal of the graph $\mathtt{Graph}_g$.  This is a closed curve, hence by Lemma \ref{repo} this curve cannot be simple.  Since $\gamma_3$ and $\mathtt{Graph}_g$ are separately simple (the former thanks to Proposition \ref{pro}), we conclude that there exists $t_0 < t, t' < t_1$ such that
$$ \gamma_3(t) = \mathtt{Graph}_g(t').$$
In particular, $\gamma_1(t), \gamma_2(t), \gamma_3(t), \gamma_4(t)$ all lie in the set $\mathtt{Graph}_{\tilde f} \cup \mathtt{Graph}_{\tilde g}$.  Since $\tilde g(t) > \tilde f(t)$, we see from \eqref{fix} that $a(t)$ and $b(t)$ cannot both vanish; thus $(\gamma_1(t), \gamma_2(t), \gamma_3(t), \gamma_4(t))$ lie in $\mathtt{Squares}$.  Now we claim that all four vertices of this square in fact lie in the set \eqref{set}.  Indeed, suppose for contradiction that one of the vertices, call it $v$, was outside of \eqref{set}, then it lies on the ray $\{ (t, f(t_0)): t < t_0\}$ or on the ray $\{ (t, f(t_1)): t > t_1 \}$.  But in either case, the set $\mathtt{Graph}_{\tilde f} \cup \mathtt{Graph}_{\tilde g} \backslash \{v\}$ is contained in the open double sector $v + \{ (x,y): |y| < |x| \}$, and hence $\mathtt{Graph}_{\tilde f} \cup \mathtt{Graph}_{\tilde g}$ cannot inscribe any square containing $v$ as a vertex (as one cannot subtend a right angle at $v$).  This implies that the set \eqref{set} inscribes a square as required, and Theorem \ref{main} follows.

\begin{remark} It is instructive to compare the above argument with the following homological argument, which requires additional regularity hypotheses on $f,g$ at the boundary points $t_0,t_1$.  Namely, suppose in addition to the hypotheses of Theorem \ref{main} that $f,g$ are differentiable at $t_0,t_1$ with $g'(t_0) > f'(t_0)$ and $g'(t_1) < f'(t_1)$; this corresponds the curve \eqref{set} being ``locally monotone'' in the sense of \cite{strom} or \cite{vz}, even at the endpoints $\mathtt{Graph}_f(t_0) = \mathtt{Graph}_g(t_0)$ and $\mathtt{Graph}_f(t_1) = \mathtt{Graph}_g(t_1)$.  A local analysis then reveals that the curve $t \mapsto \gamma_3(t)$ defined above lies in the interior of \eqref{set} for $t$ close to $t_0$, and in the exterior of \eqref{set} for $t$ close to $t_1$, and so it must cross \eqref{set} at some point; indeed, if all intersections were transversal, then it must cross this curve an odd number of times.  (Actually, it is not difficult to use the Lipschitz hypotheses to show that this curve can only cross $\mathtt{Graph}_g$ and not $\mathtt{Graph}_f$.)  In contrast, the integral argument based on the conserved integral \eqref{ident} does not give any information on the parity of crossings (indeed, if $f,g$ are not differentiable at the endpoints, one could conceivably have an infinite number of transverse crossings near the endpoints $\mathtt{Graph}_f(t_0) = \mathtt{Graph}_g(t_0)$ and $\mathtt{Graph}_f(t_1) = \mathtt{Graph}_g(t_1)$), but do not require the functions $f,g$ to be so regular at the endpoints $t_0,t_1$ that a local analysis is possible there.
\end{remark}

\begin{remark}\label{bm}  The following observations are due to Benjamin Matschke (private communcation).  The above arguments can be generalised to show that for any fixed $s,r > 0$, and with $f,g$ as in Theorem \ref{main} but with the Lipschitz constant $1-\eps$ replaced by $\tan(\alpha/2)-\eps$ with $\alpha \coloneqq \arctan(r/s) \in (0,\pi]$, the set $\mathtt{Graph}_f \cup \mathtt{Graph}_s$ inscribes a quadruple similar to the equilateral trapezoid
$$ (0,0), (1, 0), (s+1, r), (-s, r)$$
or equivalently a quadruple of the form
$$ (x,y), (x+a,y+b), (x+(s+1)a-rb, y+(s+1)b+ra), (x+(-s)a-rb, y+(-s)b+ra).$$
Theorem \ref{main} corresponds to the endpoint case $s=0,r=1$ of this more general claim.  Indeed, by repeating the above arguments one can find Lipschitz curves $\gamma_1,\gamma_2,\gamma_3,\gamma_4 \colon [t_0,t_1] \to \R^2$ of the form
\begin{align*}
\gamma_1(t) &= (t, \tilde f(t)) \\
\gamma_2(t) &= (t + a(t), \tilde f(t) + b(t)) \\
\gamma_3(t) &= (t + (s+1)a(t) - rb(t), y+(s+1)b(t) + ra(t)) \\
\gamma_4(t) &= (t + (-s) a(t) - rb(t), y+(-s)b(t) + ra(t)) 
\end{align*}
for some Lipschitz functions $a,b,\tilde f \colon [t_0,t_1] \to \R$ obeying \eqref{g1-alt}, \eqref{g2-alt}, \eqref{g4-alt}, then one can again verify that $\gamma_3$ is simple, and a variant of the calculation used to prove Lemma \ref{conserv} establishes the identity
$$ (2s+1) \int_{\gamma_1} y\ dx - (2s+1) \int_{\gamma_2} y\ dx + \int_{\gamma_3} y\ dx - \int_{\gamma_4} y\ dx = \frac{r(2s+1)}{2} \left( (a(t_1)^2 - b(t_1)^2) - (a(t_0)^2 - b(t_0)^2) \right)$$
and one can then conclude the claim by repeating the remaining arguments of this section; we leave the details to the interested reader.  On the other hand, when the equilateral trapezoid is not a rectangle or square, the known homological arguments do not seem to force the existence of an inscribed copy of the trapezoid even when the curve is smooth, because there are no symmetry reductions available to make the number of inscribed copies odd rather than even.  
\end{remark}

\section{Periodic variants of the square peg problem}

We now discuss periodic versions of the square peg problem, in which the plane $\R^2$ is replaced by the cylinder
$$ \mathtt{Cyl}_L \coloneqq (\R/L\Z) \times \R$$
for some fixed period\footnote{One could easily normalise $L$ to be $1$ if desired, but we will find it convenient to allow $L$ to be a parameter at our disposal.} $L>0$.  There is an obvious projection map $\pi_L \colon \R^2 \to \mathtt{Cyl}_L$ from the plane to the cylinder, which induces a projection $\pi_L^{\oplus 4} \colon (\R^2)^4 \to \mathtt{Cyl}_L^4$; we let $\mathtt{Squares}_L$ and $\overline{\mathtt{Squares}_L}$ be the images of $\mathtt{Squares}$ and $\overline{\mathtt{Squares}}$ under this latter projection.  More explicitly, we have
\begin{align*}
 \mathtt{Squares}_L &\coloneqq \{ ((x,y), (x+a, y+b), (x+a-b, y+a+b), (x-b, y+a)): \\
&\quad x \in \R/L\Z; y,a,b \in \R; (a,b) \neq (0,0) \}
\end{align*}
and
$$ \overline{\mathtt{Squares}_L} \coloneqq \{ ((x,y), (x+a, y+b), (x+a-b, y+a+b), (x-b, y+a)): x \in \R/L\Z; y,a,b \in \R \},$$
where we define the sum $x+a$ of an element $x \in \R/L\Z$ and a real number $a \in \R$ in the obvious fashion.
Again note that $\overline{\mathtt{Squares}_L}$ is an oriented $4$-manifold in $\mathtt{Cyl}_L^4$.  As before, a subset $\Gamma$ of $\mathtt{Cyl}_L$ is said to \emph{inscribe a square} if $\Gamma^4$ intersects $\mathtt{Squares}_L$.  We give $\mathtt{Cyl}_L$ and $\mathtt{Cyl}_L^4$ the usual flat Riemannian metric, which is then inherited by $\mathtt{Squares}_L$ and $\overline{\mathtt{Squares}_L}$.

We have a standard closed curve $\mathtt{Graph}_{0,L} \colon \R/L\Z \to \mathtt{Cyl}_L$ in $\mathtt{Cyl}_L$ defined by
$$ \mathtt{Graph}_{0,L}(t) \coloneqq (t,0);$$
one can think of $\mathtt{Graph}_{0,L}$ homologically as a $1$-cycle generating the first homology of $\mathtt{Cyl}_L$.
Any other closed curve $\gamma \colon \R/L\Z \to \mathtt{Cyl}_L$ will be homologous to $\mathtt{Graph}_{0,L}$ if and only if it takes the form
$$ \gamma(t \ \operatorname{mod}\ L) = \pi_L( \tilde \gamma(t) )$$
for some continuous lift $\tilde \gamma \colon \R \to \R^2$ of $\gamma$ that is \emph{$L\Z$-equivariant} in the sense that
\begin{equation}\label{L-equiv}
\tilde \gamma(t+L) = \tilde \gamma(t) + (L,0)
\end{equation}
for all $t \in \R$.

Amongst all the curves $\gamma$ in $\mathtt{Cyl}_L$, we isolate the \emph{polygonal curves}, in which $\gamma$ is piecewise linear (possibly after reparameterisation), that is to say $\gamma$ is the concatenation of finitely many line segments in $\mathtt{Cyl}_L$.

We now introduce the following variant of Conjecture \ref{spp}.

\begin{conjecture}[Periodic square peg problem]\label{pspp}  Let $L > 0$, and let $\sigma_1, \sigma_2 \colon \R/L\Z \to \mathtt{Cyl}_L$ be simple curves in $\mathtt{Cyl}_L$ homologous to $\mathtt{Graph}_{0,L}$.  Suppose also that the sets $\sigma_1(\R/L\Z)$ and $\sigma_2(\R/L\Z)$ are disjoint.  Then $\sigma_1(\R/L\Z) \cup \sigma_2(\R/L\Z)$ inscribes a square.
\end{conjecture}

In contrast to Conjecture \ref{spp}, we do not know the answer to Conjecture \ref{pspp} even when $\sigma_1, \sigma_2$ are smooth or piecewise polygonal (and we in fact suspect that resolving this case would soon resolve Conjecture \ref{pspp} in full generality, in analogy to Corollary \ref{poly} below).  This is because the intersection numbers produced by homological arguments become \emph{even} in the periodic setting, rather than \emph{odd}.  Of course, by rescaling we could normalise $L=1$ without loss of generality in Conjecture \ref{pspp} if desired, but we find it preferable to not enforce this normalisation.

We are able to relate Conjecture \ref{spp} to a special case of Conjecture \ref{pspp}.  To state this special case, we need a further definition:

\begin{definition}[Infinitesimally inscribed squares]  Let $L > 0$.  A closed subset $\Gamma$ of $\mathtt{Cyl}_L$ is said to \emph{inscribe infinitesimal squares} if there exists a sequence of squares
\begin{equation}\label{vert}
 S_n = ((x_n,y_n), (x_n+a_n, y_n+b_n), (x_n+a_n-b_n, y_n+a_n+b_n), (x_n-b_n, y_n+a_n)) 
\end{equation}
in $\mathtt{Squares}_L$ converging to a degenerate square $((x,y), (x,y), (x,y), (x,y))$ for some $(x,y) \in \Gamma$, such that
$$ \mathtt{dist}(S_n, \Gamma^4) = o(|a_n|+|b_n|)$$
as $n \to \infty$.
\end{definition}

Note that the property of infinitesimally inscribing squares is a purely local property: to show that a set $\Gamma$ does not infinitesimally inscribe squares, it suffices to show that for every $p \in \Gamma$, there is a set $\Gamma_p$ that agrees with $\Gamma$ in a neighbourhood of $p$ that does not infinitesimally inscribe squares.

We now give two examples of sets with the property of not inscribing infinitesimal squares.

\begin{lemma}\label{inscribe}  Let $f_1,\dots,f_k \colon \R/L\Z \to \R$ be $C$-Lipschitz functions for some $C < \tan \frac{3\pi}{8} = 1+\sqrt{2}$, such that the graphs $\mathtt{Graph}_{f_i}(\R/L\Z)$ for $i=1,\dots,k$ are all disjoint.  Then the set $\bigcup_{i=1}^k \mathtt{Graph}_{f_i}(\R/L\Z)$ does not inscribe infinitesimal squares.
\end{lemma}

\begin{proof}  By the local nature of infinitesimally inscribing squares, it suffices to show that each $\mathtt{Graph}_{f_i}(\R/L\Z)$ does not infinitesimally inscribe squares.  Suppose for contradiction that there was some $i=1,\dots,k$ and a sequence of squares \eqref{vert} with $(x_n,y_n) \to (x,y) \in \mathtt{Graph}_{f_i}(\R/L\Z)$, $(a_n,b_n) \to (0,0)$, and all vertices staying within $o(|a_n| + |b_n|)$ of $\mathtt{Graph}_{f_i}(\R/L\Z)$.  As $f_i$ is $C$-Lipschitz continuous, this implies that the eight points $\pm (a_n,b_n)$, $\pm (-b_n,a_n)$, $\pm (a_n-b_n, a_n+b_n)$, $\pm (a_n+b_n, a_n-b_n)$ all lie in the double sector $\{ (t,u): |u| \leq (C+o(1)) |t| \}$.  However, the arguments of these eight points (viewed as complex numbers) form a coset of the eighth roots of unity, while the double sector omits all the complex numbers of argument in $[\frac{3\pi}{8}, \frac{5\pi}{8}]$ if $n$ is large enough; but these two facts are in contradiction.
\end{proof}

\begin{remark} If one rotates the standard unit square $[0,1]^2$ by $\frac{\pi}{8}$, one obtains a square with the property that all its sides and diagonals have slope between $-\tan \frac{3\pi}{8}$ and $\tan \frac{3\pi}{8}$; in particular, the vertices of this square can be traversed by the graph of a $\tan \frac{3\pi}{8}$-Lipschitz function.  Gluing together infinitely many rescaled copies of this function, it is not difficult to show that the condition $C < \tan \frac{3\pi}{8}$ in Lemma \ref{inscribe} cannot be improved.
\end{remark}

\begin{lemma}\label{pat}  Let $\Gamma_1,\dots,\Gamma_k$ be disjoint simple polygonal paths in $\R^2$ (either open or closed). Then $\Gamma_1 \cup \dots \cup \Gamma_k$ does not infinitesimally inscribe squares.
\end{lemma}

\begin{proof}  Again it suffices to verify that a single $\Gamma_i$ does not inscribe squares. Suppose for contradiction that there was a sequence of squares with vertices \eqref{vert} with $(x_n,y_n) \to (x,y) \in \Gamma_i$, $(a_n,b_n) \to (0,0)$, and all vertices staying within $o(|a_n| + |b_n|)$ of $\Gamma_i$.

We can translate so that $(x,y) = (0,0)$.  The origin $(0,0)$ is either a vertex on the path $\Gamma_i$ or an interior point of an edge.  Suppose first that $(0,0)$ is an interior point of an edge, which then lies on some line $\ell$.  Then for $n$ large enough, all four vertices \eqref{vert} stay within $o(|a_n|+|b_n|)$ of this line.  Applying a suitable translation and rescaling, we can then obtain another family of squares of unit length, whose vertices \eqref{vert} are bounded and stay within $o(1)$ of $\ell$.  Using compactness to extract a limit, we obtain a unit square with all four vertices on $\ell$, which is absurd.

Now suppose that $(0,0)$ is a vertex of $\Gamma_i$, which we may take to be the origin.  This origin is the meeting point of two edges of $\Gamma_i$ that lie on two distinct lines $\ell_1, \ell_2$ passing through the origin.  If (after passing through a subsequence) all four vertices \eqref{vert} lie within $o(|a_n|+|b_n|)$ of $\ell_1$, then by rescaling and taking limits as before we obtain a unit square with all four vertices on $\ell_1$, which is absurd; similarly if all four vertices lie within $o(|a_n|+|b_n|)$ of $\ell_2$.  Thus we must have at least one vertex within $o(|a_n|+|b_n|)$ of $\ell_1$ and another within $o(|a_n|+|b_n|)$ of $\ell_2$, which forces the entire square to lie within $O(|a_n|+|b_n|)$ of the origin.  Rescaling and taking limits again, we now obtain a unit square with all four vertices on the union $\ell_1 \cup \ell_2$ of the two intersecting lines $\ell_1, \ell_2$, which is again absurd regardless of what angle $\ell_1$ and $\ell_2$ make with each other.
\end{proof}

For an example of a curve that does infinitesimally inscribe squares, one can consider any curve that has the local behaviour of a cusp such as $\{ (t^2, t^3): t \in \R \}$.

We now isolate a special case of Conjecture \ref{pspp}:

\begin{conjecture}[Periodic square peg problem, special case]\label{pspps}  Conjecture \ref{pspp} is true under the additional hypothesis that $\sigma_1(\R/L\Z) \cup \sigma_2(\R/L\Z)$ does not inscribe infinitesimal squares.
\end{conjecture}

The main result of this section is then

\begin{proposition}\label{imp}  Conjecture \ref{spp} implies Conjecture \ref{pspps}.  In particular (by Lemma \ref{pat}), Conjecture \ref{spp} implies the special case of Conjecture \ref{pspp} when the curves $\sigma_1, \sigma_2$ are polygonal paths.
\end{proposition}

To put it another way, if one wished to disprove Conjecture \ref{spp}, it would suffice to produce a union $\sigma_1(\R) \cup \sigma_2(\R)$ of two disjoint periodic curves which did not inscribe any squares or infinitesimal squares.

Proposition \ref{imp} is an immediate consequences of the following proposition:

\begin{proposition}[Transforming periodic sets to bounded sets]\label{period}  Let $\Gamma$ be a compact subset of $\mathtt{Cyl}_L$ for some $L>0$, and let $\pi_L^{-1}(\Gamma)$ be its lift to $\R^2$.  For every large natural number $n$, let $\phi_n \colon \R^2 \to \R^2$ denote the map\footnote{One can view this map as an approximation to the conformal map $z \mapsto n \operatorname{tanh} \frac{z}{n}$ in the complex plane, which ``gently pinches'' the periodic set  $\pi_L^{-1}(\Gamma)$ to a bounded set in a manner that almost preserves squares.  The use of trigonometric functions here is primarily for notational convenience; one could also use other maps than $\phi_n$ here as long as they obeyed the above-mentioned qualitative features of approximate conformality and mapping periodic sets to bounded sets.
}
$$ \phi_n(x,y) \coloneqq \left(n \operatorname{tanh} \frac{x}{n}, y \operatorname{sech}^2 \frac{x}{n}\right).$$
Then at least one of the following three statements hold.
\begin{itemize}
\item[(i)]  $\Gamma$ inscribes a square.
\item[(ii)]  $\Gamma$ inscribes infinitesimal squares.
\item[(iii)] For sufficiently large $n$, $\phi_n(\pi_L^{-1}(\Gamma)) \cup \{ (-n,0), (n,0) \}$ does not inscribe a square.
\end{itemize}
\end{proposition}

Indeed, to establish Conjecture \ref{pspps} assuming Conjecture \ref{spp}, one simply applies Proposition \ref{period} to the set $\Gamma \coloneqq \sigma_1(\R/L\Z) \cup \sigma_2(\R/L\Z)$; the conclusion (ii) is ruled out by hypothesis and the conclusion (iii) is ruled out by Conjecture \ref{spp}, leaving (i) as the only possible option.

\begin{proof}  We will assume that (iii) fails and conclude that either (i) or (ii) holds.

By hypothesis, we can find a sequence of $n$ going to infinity, and a sequence of squares with vertices \eqref{vert}, such that each square \eqref{vert} is inscribed in $\phi_n(\pi_L^{-1}(\Gamma))$.  The plan is to transform these squares into squares that either converge to a square inscribed in $\Gamma$, or become an infinitesimal inscribed square in $\Gamma$.

We first rule out a degenerate case when one of the points $(-n,0), (n,0)$ is one of the vertices \eqref{vert}.
Suppose for instance that $(x_n,y_n)$ was equal to $(n,0)$.  Since $\Gamma$ is compact, we see that $\pi_L^{-1}(\Gamma)$ is contained in a strip of the form
$$ \{ (x,y): y = O(1) \}.$$
Using the identity 
$$\operatorname{sech}^2(x) = 1 - \operatorname{tanh}^2(x) = O( 1 - |\operatorname{tanh}(x)|),$$
we conclude that $\phi_n(\pi_L^{-1}(\Gamma))$ is contained in the region
\begin{equation}\label{xiy}
 \left\{ (x,y): -n < x < n; \quad y = O\left( 1 - \frac{|x|}{n} \right)\right\}.
\end{equation}
If $(x_n,y_n) = (n,0)$ and the remaining three vertices of \eqref{vert} lie in $\phi_n(\pi_L^{-1}(\Gamma))$, this forces $(-a_n,-b_n), (-a_n+b_n,-a_n-b_n), (b_n, -a_n)$ to all have argument $O(1/n)$ when viewed as complex numbers, which is absurd for $n$ large enough, since these arguments differ by $\frac{\pi}{4}$ or $\frac{\pi}{2}$.  Thus we have $(x_n,y_n) \neq (n,0)$ (after passing to a subsequence of $n$ if necessary); a similar argument precludes any of the vertices in \eqref{vert} being equal to $(-n,0)$ or $(n,0)$.

At least one of the four vertices in \eqref{vert} must have a $y$-coordinate of magnitude at least $\frac{|a_n|+|b_n|}{2}$, since two of these $y$-coordinates differ by $a_n+b_n$ and the other two differ by $a_n-b_n$.  Applying \eqref{xiy}, we conclude that the $x$-coordinate of that vertex lies at a distance at least $c n |a_n|+|b_n|$ from $(-n,0), (n,0)$ for some $c>0$ independent of $n$; by the triangle inequality, we conclude (for $n$ large enough) that all four vertices have this property.  In particular, we have
$$ |a_n| + |b_n|= O\left( 1 - \frac{|x_n|}{n}\right).$$
Observe that in the region \eqref{xiy}, we may invert $\phi_n$ by the formula
$$ \phi_n^{-1}(x, y) \coloneqq \left(\frac{n}{2} \log \frac{n+x}{n-x}, \frac{y}{1 - x^2/n^2}\right).$$
On \eqref{xiy}, we can compute the partial derivatives
\begin{align*}
 \frac{\partial}{\partial x} \phi_n^{-1}(x,y) &= \left(\frac{1}{1-x^2/n^2}, \frac{2x}{n^2} \frac{y}{(1-x^2/n^2)^2}\right) \\
&= \frac{1}{1-x^2/n^2} \left(1, O\left(\frac{1}{n}\right)\right)
\end{align*}
and
$$ \frac{\partial}{\partial y} \phi_n^{-1}(x,y) = \left(0,\frac{1}{1-x^2/n^2}\right)$$
and so by Taylor expansion we see that
\begin{align*}
\phi^{-1}(x_n+a_n,y_n+b_n) &= \phi^{-1}(x_n,y_n) + \frac{(a_n,b_n)}{1-x_n^2/n^2} + O\left( \frac{|a_n|+|b_n|}{n (1-x_n^2/n^2)}\right ) \\
\phi^{-1}(x_n+a_n-b_n,y_n+a_n+b_n) &= \phi^{-1}(x_n,y_n) + \frac{(a_n-b_n,a_n+b_n)}{1-x_n^2/n^2} + O\left( \frac{|a_n|+|b_n|}{n (1-x_n^2/n^2)} \right) \\
\phi^{-1}(x_n-b_n,y_n+a_n) &= \phi^{-1}(x_n,y_n) + \frac{(-b_n,a_n)}{1-x_n^2/n^2} + O\left( \frac{|a_n|+|b_n|}{n (1-x_n^2/n^2)}\right ).
\end{align*}
Thus, if we set $(\tilde x_n, \tilde y_n) \coloneqq \pi_L(\phi^{-1}(x_n,y_n))$ and $(\tilde a_n, \tilde b_n) \coloneqq \frac{(a_n,b_n)}{1-x_n^2/n^2}$, we see that $\tilde a_n, \tilde b_n = O(1)$, and the four vertices of the square
$$ ((\tilde x_n, \tilde y_n), (\tilde x_n + \tilde a_n, \tilde y_n + \tilde b_n), (\tilde x_n + \tilde a_n - \tilde b_n, \tilde y_n + \tilde a_n + \tilde b_n), (\tilde x_n - \tilde b_n, \tilde y_n + \tilde a_n)) \in \mathtt{Squares}_L $$
all lie within $O(|\tilde a_n| + |\tilde b_n|/n)$ of $\Gamma$.

By passing to a subsequence, we may assume that $(\tilde a_n, \tilde b_n)$ converges to some pair $(a,b)$, which may possibly be equal to $(0,0)$.  By compactness of $\Gamma$, we may similarly assume that $(\tilde x_n, \tilde y_n)$ converges to a limit $(x,y) \in \Gamma$.  If $(a,b) \neq (0,0)$, then on taking limits using the closed nature of $\Gamma$ we conclude that the non-degenerate square 
$$ ((x,y), (x+a, y+b), (x+a-b, y+a+b), (x-b, y+a)) \in \mathtt{Squares}_L$$
is inscribed in $\Gamma$, giving (i); if instead $(a,b)=(0,0)$ then we obtain (ii).
\end{proof}

\begin{remark} In contrast to the smooth cases of Conjecture \ref{spp}, there are no homological obstructions to establishing a counterexample to Conjecture \ref{pspp}.  For instance, when the Lipschitz constants of $f,g$ are strictly less than one, the arguments of the previous section can be used to produce a quadruplet $(\gamma_1,\gamma_2,\gamma_3,\gamma_4)$ of rectifiable curves $\gamma_1,\gamma_2,\gamma_3,\gamma_4 \colon \R/L\Z \to \mathtt{Cyl}_L$ traversing squares, with $\gamma_1, \gamma_2, \gamma_4$ taking values in $\mathtt{Graph}_f(\R/L\Z)$, $\mathtt{Graph}_f(\R/L\Z)$, $\mathtt{Graph}_g(\R/L\Z)$ respectively, and all four curves homologous to the standard $1$-cycle $\mathtt{Graph}_{0,L}$.  In particular, $\gamma_3$ would (assuming sufficient transversality and regularity) intersect the graphs $\mathtt{Graph}_f(\R/L\Z)$ and $\mathtt{Graph}_g(\R/L\Z)$ an even number of times per unit period, rather than an odd number of times.  This is of course consistent with the curve not intersecting these graphs at all.  The use of an infinite oscillation to switch the parity of intersection from odd to even is reminsicent of the ``Eilenberg-Mazur swindle'' (see e.g. \cite{swindle}).
\end{remark}

\section{A quadripartite variant}\label{quad}

In Conjecture \ref{pspp}, the four vertices of the square could be distributed arbitrarily among the two graphs $\mathtt{Graph}_f(\R/L\Z)$ and $\mathtt{Graph}_g(\R/L\Z)$.  It seems to be more natural to force each vertex to lie in just one of the two graphs.  To formulate this more precisely, we introduce a further definition:

\begin{definition}[Jointly inscribing squares]\label{joint-def}  Let $L>0$.  Let $\Gamma_1, \Gamma_2, \Gamma_3, \Gamma_4$ be four sets in $\mathtt{Cyl}_L$ (possibly overlapping).  We say that the quadruplet $(\Gamma_1, \Gamma_2, \Gamma_3, \Gamma_4)$ \emph{jointly inscribes a square} if $\Gamma_1 \times \Gamma_2 \times \Gamma_3 \times \Gamma_4$ intersects $\overline{\mathtt{Squares}_L}$, or equivalently if there exist $x \in \R/L\Z$ and $y,a,b \in \R$ such that
\begin{align*}
(x,y) &\in \Gamma_1 \\
(x+a,y+b) &\in \Gamma_2 \\
(x+a-b, y+a+b) &\in \Gamma_3 \\
(x-b, y+a) &\in \Gamma_4.
\end{align*}
See Figure \ref{joint}.
\end{definition}

\begin{figure} [t]
\centering
\includegraphics[width=5in]{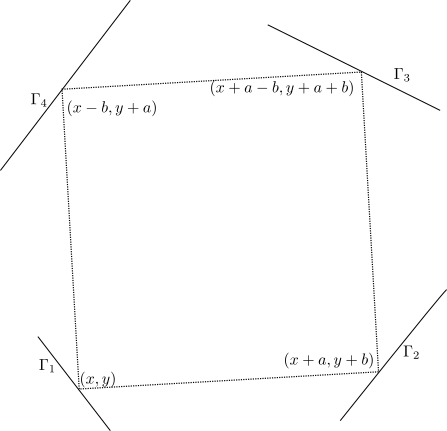}
\caption{Four line segments $(\Gamma_1, \Gamma_2, \Gamma_3, \Gamma_4)$ jointly inscribing a square.  The order (up to cyclic permutation) is important; for instance, in the given picture, the quadruple $(\Gamma_1, \Gamma_4, \Gamma_3, \Gamma_2)$ does \emph{not} jointly inscribe a square.}
\label{joint}
\end{figure}

Note in Definition \ref{joint-def} that we now permit the inscribed square to be degenerate.  Conjecture \ref{pspp} would then follow from

\begin{conjecture}[Quadripartite periodic square peg problem]\label{qpspp}  Let $L>0$, and let $\sigma_1, \sigma_2 \colon \R/L\Z \to \mathtt{Cyl}_L$ be simple closed curves homologous to $\mathtt{Graph}_{0,L}$.  Then the quadruplet
$$(\sigma_1(\R/L\Z), \sigma_1(\R/L\Z), \sigma_2(\R/L\Z), \sigma_2(\R/L\Z))$$ 
jointly inscribes a square.
\end{conjecture}

Indeed, if $\sigma_1, \sigma_2$ are as in Conjecture \ref{pspp}, we can apply Conjecture \ref{qpspp} to obtain $x \in \R/L\Z$ and $y,a,b \in \R$ with $(x,y), (x+a,y+b) \in \sigma_1(\R/L\Z)$ and $(x+a-b, y+a+b), (x-b, y+a) \in \sigma_2(\R/L\Z)$.  As $\sigma_1(\R/L\Z)$ and $\sigma_2(\R/L\Z)$ are assumed disjoint in Conjecture \ref{pspp}, we have $(a,b) \neq (0,0)$, and Conjecture \ref{pspp} follows.

We can reverse this implication in some cases:

\begin{proposition}  Conjecture \ref{pspp} implies Conjecture \ref{qpspp} in the special case that $\sigma_i(\R/L\Z)$ does not inscribe squares for $i=1,2$.
\end{proposition}

Note that the hypothesis that $\sigma_i(\R/L\Z)$ does not inscribe squares is satisfied in many cases; for instance, by modifying the proof of Lemma \ref{inscribe} we see that this is the the case if $\sigma_i$ is the graph of a $C$-Lipschitz function for some $C < \tan \frac{3\pi}{8}$.

\begin{proof}  Let $\sigma_1, \sigma_2$ be as in Conjecture \ref{qpspp}, and assume that $\sigma_1(\R/L\Z)$ and $\sigma_2(\R/L\Z)$ do not separately inscribe squares.  Let $m$ be a sufficiently large natural number, then $\sigma_1(\R/L\Z)$ and $\sigma_2(\R/L\Z) + (0,mL)$ will be disjoint.  Applying Conjecture \ref{pspp}, we may find a square
$$((x,y), (x+a, y+b), (x+a-b, y+a+b), (x-b, y+a)) \in \mathtt{Squares}_L$$
inscribed in $\sigma_1(\R/L\Z) \cup (\sigma_2(\R/L\Z) + (0,mL))$.  In particular we have
$$ y, y+b, y+a+b, y+a \in [-C,C] \cup [mL-C, mL+C]$$
for some $C>0$ independent of $m$.  If $m$ is large enough, this forces the quadruple $(y,y+b,y+a+b,y+a)$ to be of the form $(O(C), O(C), O(C), O(C))$, the form $(mL+O(C), mL+O(C), mL+O(C), mL+O(C))$, or some cyclic permutation of $(O(C), O(C), mL+O(C), mL+O(C))$.  In the first case, we have a square inscribed in $\sigma_1(\R/L\Z)$, and in the second case we have (after translation by $(0,mL)$) a square inscribed by $\sigma_2(\R/L\Z)$; both these cases are ruled out by hypothesis.  Thus, after cyclic permutation, we may assume that
$$ (y, y+b, y+a+b, y+a) = (O(C), O(C), mL+O(C), mL+O(C))$$
which implies that the (possibly degenerate) square
$$((x,y), (x+a-mL, y+b), (x+a-b-mL, y+a+b), (x-b, y+a-mL)) \in \overline{\mathtt{Squares}_L}$$
is jointly inscribed by $(\sigma_1(\R/L\Z), \sigma_1(\R/L\Z), \sigma_2(\R/L\Z), \sigma_2(\R/L\Z))$, giving Conjecture \ref{qpspp} in this case.
\end{proof}

Because the squares in Definition \ref{joint-def} are now permitted to be degenerate, Conjecture \ref{qpspp} enjoys good convergence properties with respect to limits:

\begin{proposition}[Stability of not jointly inscribing squares]\label{lim}  Let $L > 0$.  Let $\sigma_{1,n}, \sigma_{2,n} \colon \R/L\Z \to \mathtt{Cyl}_L$ be sequences of simple closed curves which converge uniformly to simple closed curves $\sigma_1, \sigma_2 \colon \R/L\Z \to \mathtt{Cyl}_L$ respectively as $n \to \infty$.  If each of the quadruples 
$$(\sigma_{1,n}(\R/L\Z), \sigma_{1,n}(\R/L\Z), \sigma_{2,n}(\R/L\Z), \sigma_{2,n}(\R/L\Z))$$ 
jointly inscribe a square, then so does 
$$(\sigma_1(\R/L\Z), \sigma_1(\R/L\Z), \sigma_2(\R/L\Z), \sigma_2(\R/L\Z)).$$
\end{proposition}

It is possible to weaken the requirement of uniform convergence (for instance, one can just assume pointwise convergence if the curves $\sigma_{1,n}, \sigma_{2,n}$ are uniformly bounded), but we will not need to do so here.

\begin{proof}  By hypothesis, one can find a sequence $p_n \in \overline{\mathtt{Squares}_L}$ such that $p_n \in \sigma_{1,n}(\R/L\Z) \times \sigma_{1,n}(\R/L\Z) \times \sigma_{2,n}(\R/L\Z) \times \sigma_{2,n}(\R/L\Z)$ for all $n$.  As $\sigma_{1,n}, \sigma_{2,n}$ converge uniformly to $\sigma_1,\sigma_2$, the $p_n$ are bounded and thus have at least one limit point $p$, which must lie in both $\overline{\mathtt{Squares}_L}$ and $\sigma_{1}(\R/L\Z) \times \sigma_{1}(\R/L\Z) \times \sigma_{2}(\R/L\Z) \times \sigma_{2}(\R/L\Z)$, giving the claim.
\end{proof}

As one application of this proposition, we have

\begin{corollary}\label{poly}  In order to prove Conjecture \ref{qpspp}, it suffices to do so in the case that the simple closed curves $\sigma_1, \sigma_2 \colon \R/L\Z \to \mathtt{Cyl}_L$ are polygonal paths.
\end{corollary}

This can be compared with the situation with Conjecture \ref{spp}, which is known to be true for polygonal paths, but for which one cannot take limits to conclude the general case, due to the possibility of the inscribed squares degenerating to zero.

\begin{proof}  By Proposition \ref{lim}, it suffices to show that any simple closed curve $\gamma \colon \R/L\Z \to \mathtt{Cyl}_L$ can be uniformly approximated to any desired accuracy $O(\eps)$ for $\eps>0$ by a simple polygonal closed curve $\tilde \gamma \colon \R/L\Z \to \mathtt{Cyl}_L$ (note from a winding number argument that if $\gamma$ is homologous to $\mathtt{Graph}_{0,L}$, then $\tilde \gamma$ will be also if $\eps$ is small enough).  By uniform continuity, there exists a natural number $N$ such that $d_{\mathtt{Cyl}_L}(\gamma(t),\gamma(t')) \leq \eps$ whenever $d_{\R/L\Z}(t,t') \leq L/N$, where $d_{\mathtt{Cyl}_L}, d_{\R/L\Z}$ denote the Riemannian distance functions on $\mathtt{Cyl}_L, \R/L\Z$ respectively; as $\gamma$ is simple and continuous, a compactness argument shows that there also exists $0 < \delta < \eps$ such that $d_{\mathtt{Cyl}_L}(\gamma(t), \gamma(t')) > 4 \delta$ whenever $d_{\R/L\Z}(t,t') \geq L/N$.   Finally, by uniform continuity again, there exists a natural number $M \geq N$ such that $d_{\mathtt{Cyl}_L}(\gamma(t), \gamma(t')) \leq \delta$ whenever $d_{\R/L\Z}(t,t') \leq L/M$.

Let $\gamma_1 \colon \R/L\Z \to \R^2$ be the polygonal path such that $\gamma_1(jL/M) \coloneqq \gamma(jL/M)$ for every $j \in \Z/M\Z$, with $\gamma_1$ linear on each interval $jL/M + [0, 1/M]$.  From the triangle inequality we see that $d_{\mathtt{Cyl}_L}(\gamma_1(t), \gamma(t)) \leq 2\delta$ for all $t \in \R$.  Unfortunately, $\gamma_1$ need not be simple.  However, if $t,t'$ are such that $\gamma_1(t) = \gamma_1(t')$, then by the triangle inequality we have $d_{\mathtt{Cyl}_L}(\gamma(t), \gamma(t')) \leq 4\delta$, and hence $d_{\R/L\Z}(t,t') \leq L/N$.  Using a greedy algorithm to iteratively remove loops from the polygonal path $\gamma_1$ (with each removal decreasing the number of edges remaining in $\gamma_1$ in a unit period), we may thus find a finite family $I_1,\dots,I_k$ of disjoint closed intervals in $\R/L\Z$, each of length at most $L/N$, such that paths $\gamma_1|_{I_j} \colon I_j \to \mathtt{Cyl}_L$ is closed for each $1 \leq j \leq k$ (i.e. $\gamma_1$ evaluates to the same point at the left and right endpoints of $I_j$), and such that $\gamma_1$ becomes simple once each of the intervals $I_j$ is contracted to a point.  Note also that all of the loops of $\gamma_1$ removed by this process have diameter $O(\eps)$.  If one then chooses a sufficiently small neighbourhood interval $\tilde I_j$ for each $I_j$, and defines $\tilde \gamma$ to equal $\gamma_1$ outside $\bigcup_{j=1}^k \tilde I_j$, and linear on each of the $\tilde I_j$, then we see that $\tilde \gamma$ is a simple $L\Z$-equivariant polygonal path that lies within $O(\eps)$ of $\gamma$, as required.
\end{proof}

We in fact believe the following stronger claim than Conjecture \ref{qpspp} to hold:

\begin{conjecture}[Area inequality]\label{area-ineq}  Let $L>0$, and let $\sigma_{1}, \sigma_{2}, \sigma_{3}, \sigma_{4} \colon \R/L\Z \to \mathtt{Cyl}_L$ be simple closed polygonal paths homologous to $\mathtt{Graph}_{0,L}$.  If $(\sigma_{1}(\R/L\Z), \sigma_{2}(\R/L\Z), \sigma_{3}(\R/L\Z), \sigma_{4}(\R/L\Z))$ does not jointly inscribe a square, then
\begin{equation}\label{area-claim}
\int_{\sigma_1} y\ dx - \int_{\sigma_2} y\ dx + \int_{\sigma_3} y\ dx - \int_{\sigma_4} y\ dx \neq 0.
\end{equation}
\end{conjecture}

Note that the $1$-form $y\ dx$ is well defined on $\mathtt{Cyl}_L$ and $\sigma_1,\sigma_2,\sigma_3,\sigma_4$ can be viewed as $1$-cycles, so the integrals in \eqref{area-claim} make sense as integration of differential forms (but one could also use Definition \ref{area-def} with the obvious modifications if desired).  One can also simplify the left-hand side of \eqref{area-claim} as
\begin{equation}\label{integra}
\int_{\sigma_1 - \sigma_2 + \sigma_3 -\sigma_4} y\ dx
\end{equation}
where $\sigma_1 - \sigma_2 + \sigma_3 -\sigma_4$ is interpreted as a $1$-cycle.  Viewed contrapositively, Conjecture \ref{area-ineq} then asserts that $(\sigma_{1}(\R/L\Z), \sigma_{2}(\R/L\Z), \sigma_{3}(\R/L\Z), \sigma_{4}(\R/L\Z))$ must inscribe a square whenever the integral \eqref{integra} vanishes.

Clearly, Conjecture \ref{qpspp} follows from Conjecture \ref{area-ineq} by first using Corollary \ref{poly} to reduce to the case where $\sigma_1,\sigma_2$ are polygonal paths, and then applying Conjecture \ref{area-ineq} with $(\sigma_1,\sigma_2,\sigma_3,\sigma_4)$ replaced by $(\sigma_1,\sigma_1,\sigma_2,\sigma_2)$.

Conjecture \ref{area-ineq} is somewhat strong compared with the other conjectures in this paper.  Nevertheless, we can repeat the arguments in Section \ref{pos-sec} to obtain some evidence for it:

\begin{theorem}\label{ook}  Conjecture \ref{area-ineq} holds when $\sigma_{2} = \mathtt{Graph}_{f_2}$ and $\sigma_{4} = \mathtt{Graph}_{f_4}$ for some $(1-\eps)$-Lipschitz functions $f_2,f_4 \colon \R/L\Z \to \mathtt{Cyl}_L$ and some $\eps>0$.
\end{theorem}

Of course, by cyclic permutation one can replace the role of $\sigma_2,\sigma_4$ here by $\sigma_1,\sigma_3$.

\begin{proof}  Write $\sigma_1(t) \coloneqq (x_1(t), y_1(t))$.  For any $t \in \R/L\Z$, the map
$$ (a,b) \mapsto \left( f_4(x_1(t)-b) - y_1(t), f_2(x_1(t)+a) - y_1(t)\right) $$
is a contraction on $\R^2$ with constant at most $1-\eps$, hence by the contraction mapping theorem there exist unique continuous functions $a,b \colon \R/L\Z \to \R$ such that
$$ (a(t),b(t)) = \left( f_4(x_1(t)-b(t)) - y_1(t), f_2(x_1(t)+a(t)) - y_1(t)\right)$$
for all $t \in \R/L\Z$.  As $\sigma_1,\sigma_2,\sigma_4$ are polygonal paths, $f_2,f_4$ are piecewise linear functions, which implies that $a,b$ are also piecewise linear.  We then set $\gamma_1,\gamma_2,\gamma_3,\gamma_4 \colon \R/L\Z \to \mathtt{Cyl}_L$ to be the polygonal paths
\begin{align*}
\gamma_1(t) &\coloneqq \sigma_1(t) \\
\gamma_2(t) &\coloneqq \sigma_1(t) + (a(t), b(t)) \\
&= \mathtt{Graph}_{f_2}(x_1(t)+a(t))\\
\gamma_3(t) &\coloneqq \sigma_1(t) + (a(t)-b(t), a(t)+b(t)) \\
\gamma_4(t) &\coloneqq \sigma_1(t) + (-b(t), a(t)) \\
&= \mathtt{Graph}_{f_4}(x_1(t)-b(t)).
\end{align*}
Clearly $(\gamma_1,\gamma_2,\gamma_3,\gamma_4)$ traverses squares.  Applying Lemma \ref{conserv}, we conclude that
$$ \int_{\gamma_1} y\ dx - \int_{\gamma_2} y\ dx + \int_{\gamma_3} y\ dx - \int_{\gamma_4} y\ dx = 0.$$
Since $t \mapsto \mathtt{Graph}_{f_2}(x_1(t)+a(t))$ and $t \mapsto \mathtt{Graph}_{f_4}(x_1(t)-b(t))$ are homologous to $\sigma_2$ and $\sigma_4$
respectively, and all $1$-forms on curves such as $\mathtt{Graph}_{f_2}(\R/L\Z)$ or $\mathtt{Graph}_{f_4}(\R/L\Z)$ are automatically closed, we have $\int_{\gamma_i} y\ dx = \int_{\sigma_i} y\ dx$ for $i=1,2,4$.  Thus it suffices to show that
\begin{equation}\label{nul}
 \int_{\gamma_3 - \sigma_3} y\ dx \neq 0.
\end{equation}
The argument in Proposition \ref{pro} (unwrapping the curves from $\mathtt{Cyl}_L$ to $\R^2$) shows that $\gamma_3$ is simple, and the graph $\sigma_{3}$ is of course also simple.  As we are assuming $(\sigma_{1}(\R/L\Z), \sigma_{2}(\R/L\Z), \sigma_{3}(\R/L\Z), \sigma_{4}(\R/L\Z))$ to not jointly inscribe squares, $\gamma_3(\R/L\Z)$ and $\sigma_{3}(\R/L\Z)$ must stay disjoint.  The curve $\gamma_3$ is homotopic to $\sigma_1$ and thus also homologous to $\mathtt{Graph}_{0,L}$.  Applying the Jordan curve theorem, we see that the closed polygonal paths $\gamma_3, \sigma_3$ enclose some non-empty polygonal region $\Omega$ in $\mathtt{Cyl}_L$.  By Stokes' theorem, we conclude that the left-hand side of \eqref{nul} is equal to some sign times the Lebesgue measure of $\Omega$, giving \eqref{nul} as required.
\end{proof}

We also have an analogue of Corollary \ref{poly}:

\begin{proposition}[Stability of area inequality]\label{lim2}  Let $L > 0$.  Suppose that $\sigma_{1,n},\sigma_{2,n}, \sigma_{3,n}, \sigma_{4,n} \colon \R/L\Z \to \mathtt{Cyl}_L$ are simple closed polygonal paths which converge uniformly to simple closed polygonal paths $\sigma_{1}, \sigma_{2}, \sigma_{3}, \sigma_{4} \colon \R/L\Z \to \mathtt{Cyl}_L$ respectively.  If, for all sufficiently small $h \in \R$ Conjecture \ref{area-ineq} holds for each of the quadruples $(\sigma_{1,n},\sigma_{2,n}, \sigma_{3,n}, \sigma_{4,n} + (0,h))$ for all $n \geq 1$, then it also holds for $(\sigma_{1}, \sigma_{2}, \sigma_{3}, \sigma_{4})$.
\end{proposition}

This proposition will be useful for placing the polygonal paths $\sigma_{1}, \sigma_{2}, \sigma_{3}, \sigma_{4}$ in ``general position''.

\begin{proof}  We can of course assume that $(\sigma_{1}(\R/L\Z), \sigma_{2}(\R/L\Z), \sigma_{3}(\R/L\Z), \sigma_{4}(\R/L\Z))$ does not jointly inscribe a square, as the claim is trivial otherwise.  By Proposition \ref{lim} applied in the contrapositive, we see that there exists an $\eps>0$ and $N \geq 1$ such that $(\sigma_{1,n}(\R/L\Z), \sigma_{1,n}(\R/L\Z), \sigma_{1,n}(\R/L\Z), \sigma_{1,n}(\R) + (0,h))$ does not jointly inscribe squares for any $|h| \leq \eps$ and $n \geq N$.  Applying the hypothesis (and shrinking $\eps$ if necessary), we conclude that
$$ \int_{\sigma_{1,n} - \sigma_{2,n} + \sigma_{3,n} - \sigma_{4,n}} y\ dx - Lh \neq 0$$
for $n \geq N$ and $|y| \leq \eps$, and hence
$$ \left|\int_{\sigma_{1,n} - \sigma_{2,n} + \sigma_{3,n} - \sigma_{4,n}} y\ dx\right| \geq L \eps;$$
taking limits as $n \to \infty$ we conclude that
$$ \left|\int_{\sigma_{1} - \sigma_{2} + \sigma_{3} - \sigma_{4}} y\ dx\right| \geq L \eps$$
giving the claim.
\end{proof}

In the remainder of this section, we discuss how to interpret\footnote{This discussion is not used elsewhere in the paper and may be safely skipped by the reader if desired.} the area inequality conjecture (Conjecture \ref{area-ineq}) using the language of homology.  Let $L > 0$, and let $\sigma_{1}, \sigma_{2}, \sigma_{3}, \sigma_{4} \colon \R/L\Z \to \mathtt{Cyl}_L$ be simple closed polygonal paths.  We say that $\sigma_1,\dots,\sigma_4$ are in \emph{general position} if the following hold for any distinct $i,j,k \in \{1,2,3,4\}$:
\begin{itemize}
\item[(i)]  For any edge $e$ of $\sigma_i$ and any edge $f$ of $\sigma_j$, the angle between the direction of $e$ and the direction of $f$ is not an integer multiple of $\frac{\pi}{4}$.
\item[(ii)]  For any vertices $u,v,w$ of $\sigma_i, \sigma_j, \sigma_k$ respectively, there does not exist a square with $u,v,w$ as three of its four vertices.
\end{itemize}
It is easy to see that one can perturb $\sigma_1,\sigma_2,\sigma_3,\sigma_4$ by an arbitrarily small amount to be in general position (e.g. a random perturbation will almost surely work); furthermore one can ensure that this general position will persist even after shifting $\sigma_4$ vertically by an arbitrary amount.  Hence by Proposition \ref{lim2}, to prove Conjecture \ref{area-ineq} it suffices to do so under the hypothesis of general position. 

The Cartesian product $\sigma_1 \times \sigma_2 \times \mathtt{Cyl}_L \times \sigma_4$ can be viewed as a (polyhedral) $5$-cycle in $\mathtt{Cyl}_L^4$; by the hypotheses of general position, this cycle intersects the oriented $4$-manifold $\overline{\mathtt{Squares}_L}$ transversely (and in a compact set), giving rise to a $1$-cycle $\sigma_{124}$ in $\overline{\mathtt{Squares}_L}$.  For $j=1,2,3,4$, let $\pi_j \colon \overline{\mathtt{Squares}_L} \to \mathtt{Cyl}_L$ be the projection map to the $j^{\operatorname{th}}$ coordinate.  Because the $5$-cycle $\sigma_1 \times \sigma_2 \times \mathtt{Cyl}_L \times \sigma_4$ is homologous to the $5$-cycle $\mathtt{Graph}_{0,L} \times \mathtt{Graph}_{0,L} \times \mathtt{Cyl}_L \times \mathtt{Graph}_{0,L}$, we see on restricting to $\overline{\mathtt{Squares}_L}$ that the $1$-cycle $\sigma_{124}$ is homologous to the $1$-cycle
\begin{equation}\label{grad}
 \mathtt{Graph}_{0,L}^\Delta \coloneqq \{ (p,p,p,p): p \in \mathtt{Graph}_{0,L}(\R/L\Z) \}
\end{equation}
(with the usual orientation), and hence the pushforwards $\gamma_j \coloneqq (\pi_j)_* \sigma_{124}$ (which are polygonal $1$-cycles on $\mathtt{Cyl}_L$) are homologous to $\mathtt{Graph}_{0,L}$ and thus to $\sigma_j$ for $j=1,2,3,4$.  For $j=1,2,4$, $\gamma_j$ takes values in the curve $\sigma_j(\R/L\Z)$; as $y\ dx$ is closed on that curve, we thus have
$$ \int_{\sigma_{124}} \pi_j^*(y\ dx) = \int_{(\pi_j)_* \sigma_{124}} y\ dx = \int_{\sigma_j} y\ dx$$
for $j=1,2,4$.  On the other hand, from Remark \ref{homog} (adapted to the cylinder $\mathtt{Cyl}_L$ in the obvious	 fashion), the $1$-form
\begin{equation}\label{exact}
 \pi_1^*(y\ dx) - \pi_2^*(y\ dx) + \pi_3^*(y\ dx) - \pi_4^*(y\ dx)
\end{equation}
is exact on $\overline{\mathtt{Squares}_L}$, and hence
$$ \int_{\sigma_{124}} \pi_1^*(y\ dx) - \pi_2^*(y\ dx) + \pi_3^*(y\ dx) - \pi_4^*(y\ dx) = 0.$$
Putting all this together, we see that the claim \eqref{area-claim} can be rewritten as
\begin{equation}\label{dill}
 \int_{\gamma_3 - \sigma_{3}} y\ dx \neq 0.
\end{equation}
Meanwhile, the hypothesis that $(\sigma_1(\R/L\Z), \sigma_2(\R/L\Z), \sigma_3(\R/L\Z),\sigma_4(\R/L\Z))$ do not jointly inscribe squares is equivalent to the assertion that the $1$-cycles $\gamma_3$ and $\sigma_3$ are disjoint.

The $4$-manifold $\overline{\mathtt{Squares}_L}$ is homeomorphic to $\R/L\Z \times \R^3$, and so its first homology is generated by $\mathtt{Graph}_{0,L}^\Delta$.  One can decompose the $1$-cycle $\sigma_{124}$ as an integer linear combination of finitely many closed polygonal curves in $\overline{\mathtt{Squares}_L}$ (which are allowed to intersect each other); as $\sigma_{124}$ is homologous to $\mathtt{Graph}_{0,L}^\Delta$, one of these curves, call it $\sigma_{124}^0 \colon \R/L\Z \to \overline{\mathtt{Squares}_L}$, must be homologous to $m \mathtt{Graph}_{0,L}^\Delta$ for some nonzero integer $m$, thus it obeys the equivariance $\sigma_{124}^0(t+L) = \gamma_{124}^0(t) + (mL,0)$.  By reversing orientation we may assume $m$ is positive.  

We now lift the cylinder $\mathtt{Cyl}_L$ up to the larger cylinder $\mathtt{Cyl}_{mL}$, which is an $m$-fold cover of the original cylinder; one can similarly lift $\overline{\mathtt{Squares}_L}$ to the $m$-fold cover $\overline{\mathtt{Squares}_{mL}}$.  The curves $\sigma_j: \R/L\Z \to \mathtt{Cyl}_L$ lift to curves $\tilde \sigma_j: \R/mL\Z \to \mathtt{Cyl}_{mL}$.  The $1$-cycle $\sigma_{124}$ lifts to a $1$-cycle $\tilde \sigma_{124}$ homologous to $\mathtt{Graph}_{0,mL}^\Delta$; meanwhile, the curve $\sigma_{124}^0$ lifts to $m$ copies of a curve $\tilde \sigma_{124}^0$ homologous to $\mathtt{Graph}_{0,mL}^\Delta$ (and parameterised by $\R/mL\Z$), and contained (as a set) in $\tilde \sigma_{124}$.  We conclude that $\tilde \sigma_{124} - \tilde \sigma_{124}^0$ is a $1$-boundary, thus
$$ \tilde \sigma_{124} = \tilde \sigma_{124}^0 + \partial U$$
for some $2$-cycle $U$ in $\overline{\mathtt{Squares}_{mL}}$.  We can then define curves $\tilde \gamma_j^0 \colon \R/mL\Z \to \mathtt{Cyl}_{mL}$ for $j=1,2,3,4$ by $\tilde \gamma_j^0 \coloneqq \pi_j \circ \tilde \sigma_{124}^0$; these curves are the analogues of the curves $\gamma_1,\gamma_2,\gamma_3,\gamma_4$ from Section \ref{pos-sec}.  As the $1$-form $y\ dx$ is closed on the curves $\sigma_j(\R/L\Z)$, as well as their lifts $\tilde \sigma_j(\R/mL\Z)$ to $\mathtt{Cyl}_{mL}$, we have
$$ \int_{\partial U} \pi_j^*(y\ dx) = \int_{\partial (\pi_j)_* U} y\ dx = 0$$
for $j=1,2,4$, while from the exact nature of \eqref{exact} gives
$$ \int_{\partial U} \pi_1^*(y\ dx) - \pi_2^*(y\ dx) + \pi_3^*(y\ dx) - \pi_4^*(y\ dx) = 0$$
and hence
$$ \int_{\partial (\pi_3)_* U} y\ dx = 0.$$
Hence one can also express \eqref{dill} as
\begin{equation}\label{dill-2}
 \int_{\tilde \gamma_3^0 - \tilde \sigma_{3}} y\ dx \neq 0.
\end{equation}
The $1$-cycle $\tilde \gamma_3^0 - \tilde \sigma_{3}$ is homologous to $\mathtt{Graph}_{0,mL} - \mathtt{Graph}_{0,mL} = 0$ and thus can be expressed as a $1$-boundary
$\tilde \gamma_3^0 - \tilde \sigma_3 = \partial \Omega$ for some $2$-cycle $\Omega$ in $\mathtt{Cyl}_{mL}$.  By Stokes' theorem, \eqref{dill-2} can now be expressed as
\begin{equation}\label{dill-3}
 -\int_{\Omega}\ dx \wedge dy \neq 0.
\end{equation}
In the case when the $\sigma_2, \sigma_4$ were graphs of Lipschitz functions of constant less than $1$, the closed path $\tilde \gamma_{3}^0$ was necessarily simple (and one could take $m=1$); if $(\sigma_1(\R/L\Z),\sigma_2(\R/L\Z),\sigma_3(\R/L\Z),\sigma_4(\R/L\Z))$ did not jointly inscribe squares, then $\tilde \gamma_3^0$ avoided $\sigma_3$, and so by the Jordan curve theorem the $2$-cycle $\Omega$ had a definite sign which yielded \eqref{dill-3} and thus \eqref{dill-2}, \eqref{dill}, \eqref{area-claim}.  Unfortunately, in the general case it is possible for the $2$-cycle $\Omega$ to contain both positive and negative components, even after stripping out the $1$-boundaries $\partial U$ from $\tilde \sigma_{124}$ and working just with $\tilde \sigma^0_{124}$.  However, from working with numerous examples, it appears to the author that there is always an imbalance between the positive and negative components of $\Omega$ that leads to the inequality \eqref{dill-3} and hence to Conjecture \ref{area-ineq}.  Unfortunately, the author was unable to locate an argument to establish this claim rigorously.

\begin{remark}  The Jordan curve theorem does imply that the simple closed curve $\tilde \sigma_{3}$ partitions the cylinder $\mathtt{Cyl}_{mL}$ into two connected components, the region ``above'' $\tilde \sigma_{3}$ (which contains all the points in $\mathtt{Cyl}_{mL}$ with sufficiently large positive $y$ coordinate) and the region ``below'' $\tilde \sigma_{3}$ (which contains all the points in $\mathtt{Cyl}_{mL}$ with sufficiently large negative $y$ coordinate).  Suppose that $(\sigma_1(\R/L\Z),\sigma_2(\R/L\Z), \sigma_3(\R/L\Z), \sigma_4(\R/L\Z))$ does not jointly inscribe a square, so that $\gamma_3$ avoids $\sigma_3$.  
If we let $K$ denote the connected component of $\tilde \gamma_3$ (viewed as a subset of $\mathtt{Cyl}_L$) that contains (the image of) $\tilde \gamma_{3}^0$, then $K$ must then either lie entirely in the region above $\tilde \sigma_{3}$ or the region below $\tilde \sigma_{3}$.  We conjecture that this determines the sign in \eqref{area-claim} (or \eqref{dill}, \eqref{dill-2}, \eqref{dill-3}).  Namely, if $K$ is in the region above $\tilde \sigma_{3}$, we conjecture that left-hand side of \eqref{area-claim} (or \eqref{dill}, \eqref{dill-2}, \eqref{dill-3}) must be strictly positive, and if $K$ is in the region below $\tilde \sigma_{3}$, then these left-hand sides must be strictly negative.  An equivalent statement of this is that if $\tilde \gamma^+_{3}, \tilde \gamma^-_{3} \colon \R/mL\Z \to \mathtt{Cyl}_{mL}$ are simple closed polygonal paths that traverse the upper and lower boundary of $K$ respectively (by which we mean the portions of the boundary of $K$ that can be connected by a path to points in $\mathtt{Cyl}_{mL}$ with arbitrarily large positive or negative $y$ coordinate respectively), then we have the inequalities
\begin{equation}\label{areo}
\int_{\tilde \gamma^-_3} y\ dx \leq \int_{\tilde \gamma^0_3} y\ dx \leq \int_{\tilde \gamma^+_3} y\ dx.
\end{equation}
This claim, which implies Conjecture \ref{area-ineq}, is an assertion about the relative sizes of the ``holes'' in $K$; see Figure \ref{Kfig}.
\end{remark}

\begin{figure} [t]
\centering
\includegraphics[width=5in]{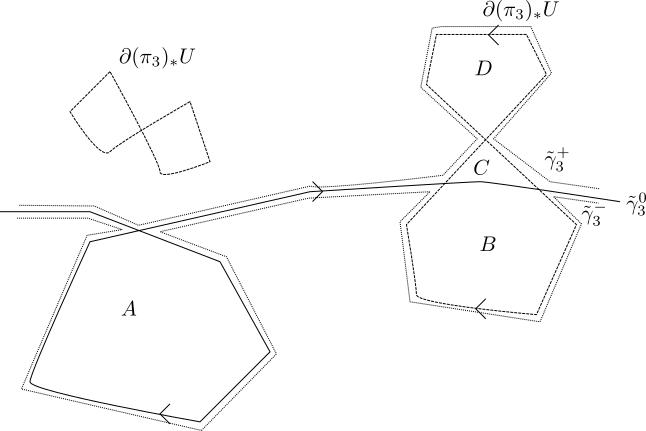}
\caption{A polygonal path $\tilde \gamma_{3}^0$ (drawn as a solid line), together with some additional $1$-boundaries $\partial (\pi_3)_* U$ (the two dashed lines).  Here, $K$ is $\tilde \gamma_{3}^0$ together with the component of $\partial (\pi_3)_* U$ that intersects $\tilde \gamma_{3}^0$.  The paths $\tilde \gamma_{3}^-$ and $\tilde \gamma_{3}^+$ are drawn as dotted lines; they have been moved slightly away from $\gamma_3^0$ for visibility.  The area inequalities \eqref{areo} can then be written as $0 \leq 2 |A| + |B| \leq |A|+|B|+|C|+|D|$, where $|A|$ denotes the Lebesgue measure of the region $A$ in the figure, and similarly for $|B|, |C|, |D|$.  Each $1$-boundary gives a zero contribution to the area under $\tilde \gamma_3$, so one also has $|B|+|C|=|D|$.  In the depicted scenario, the first area inequality is automatically true, but the second one is not necessarily so.}
\label{Kfig}
\end{figure}

\section{A combinatorial variant}

Conjecture \ref{area-ineq} appears difficult to resolve in general. However, there is a more tractable-seeming special case of Conjecture \ref{area-ineq} which captures many of the key features of the full conjecture:

\begin{conjecture}[Special case of area inequality]\label{area-1d}  Conjecture \ref{area-ineq} holds when $\sigma_1 = \mathtt{Graph}_{0,L}$.
\end{conjecture}  

Once one makes the restriction $\sigma_1 = \mathtt{Graph}_{0,L}$, Conjecture \ref{area-ineq} turns out to collapse from a two-dimensional problem to a more tractable one-dimensional one.  Indeed, suppose that the tuple $(\sigma_1(\R/L\Z), \sigma_2(\R/L\Z), \sigma_3(\R/L\Z), \sigma_4(\R/L\Z))$ does not jointly inscribing a square, with $\sigma_1 = \mathtt{Graph}_{0,L}$.  That is to say, there does not exist $x \in \R/L\Z$ and $y,a,b \in \R$ with
\begin{align*}
(x, y) &\in \sigma_1(\R/L\Z) \\
(x+a, y+b) &\in \sigma_2(\R/L\Z) \\
(x+a-b, y+a+b) &\in \sigma_3(\R/L\Z) \\
(x-b, y+a) &\in \sigma_4(\R/L\Z).
\end{align*}
The first condition $(x,y) \in \sigma_1(\R/L\Z)$ simply asserts that $y=0$.  If one now defines the linearly transformed closed polygonal curves $\tilde \sigma_2, \tilde \sigma_3, \tilde \sigma_4 \colon \R/L\Z \to \mathtt{Cyl}_L$ by the formulae
\begin{align*}
\tilde \sigma_2(t) &= (x_2(t) - y_2(t), y_2(t)) \\
\tilde \sigma_3(t) &= (x_3(t), -y_3(t)) \\
\tilde \sigma_4(t) &= (x_4(t) + y_4(t), y_4(t))
\end{align*}
where $x_i, y_i \colon \R \to \R$ are the components of $\sigma_i$ for $i=1,2,3,4$, then $\tilde \sigma_2, \tilde \sigma_3, \tilde \sigma_4$ remain simple, and (on setting $\tilde x \coloneqq x+a-b$) we see that the property of the quadruple
$$(\sigma_1(\R/L\Z), \sigma_2(\R/L\Z), \sigma_3(\R/L\Z), \sigma_4(\R/L\Z))$$ 
not jointly inscribing squares is equivalent to the non-existence of real numbers $\tilde x, a, b$ such that
\begin{align*}
(\tilde x, b) &\in \tilde \sigma_2(\R/L\Z) \\
(\tilde x, -a-b) &\in \tilde \sigma_3(\R/L\Z) \\
(\tilde x, a) &\in \tilde \sigma_4(\R/L\Z).
\end{align*}
Also, from the change of variables we see that
$$\int_{\tilde \sigma_j} y\ dx = (-1)^j \int_{ \sigma_j} y\ dx$$
for $j=2,3,4$.
Relabeling $\tilde \sigma_2, \tilde \sigma_3, \tilde \sigma_4$ as $\gamma_1, \gamma_2, \gamma_3$, and writing $b,-a-b,a$ as $y_1,y_2,y_3$ respectively, we thus see that Conjecture \ref{area-1d} is equivalent to the following more symmetric version:

\begin{conjecture}[Special case of area inequality, symmetric form]\label{sai}  Let $L>0$, and let $\gamma_1, \gamma_2, \gamma_3 \colon \R/L\Z \to \mathtt{Cyl}_L$ be simple closed polygonal paths homologous to $\mathtt{Graph}_{0,L}$, such that there does not exist points $(x,y_i) \in \gamma_i(\R/L\Z)$ with $x,y_1,y_2,y_3 \in \R$ and $y_1+y_2+y_3 = 0$.  Then one has
$$ \int_{\gamma_1+\gamma_2+\gamma_3} y\ dx \neq 0.$$
\end{conjecture}

In this section we show that Conjecture \ref{sai} is equivalent to an almost purely combinatorial statement.  To formulate it, we need some definitions.  Recall that the \emph{signum function} $\operatorname{sgn}: [-\infty,+\infty] \to \{-1,+1\}$ on the extended real line $[-\infty,+\infty]$ is defined to equal $-1$ of $[-\infty,0)$, $0$ on $0$, and $+1$ on $(0,+\infty]$.

\begin{definition}[Non-crossing sums]\label{ncs-def}  Let $m=2,3$, and for each $1 \leq i \leq m$, let $y_{i,1}, y_{i,2} \in [-\infty,+\infty]$ be extended reals.  We say that the pairs $\{ y_{i,1},  y_{i,2}\}$ for $i=1,\dots,m$ have \emph{non-crossing sums} if the following axioms are obeyed:
\begin{itemize}
\item[(i)]  Either all of the $y_{i,1}, y_{i,2}$ avoid $+\infty$, or they all avoid $-\infty$. 
\item[(ii)]  For any $j_1,\dots,j_m \in \{1,2\}$, the sum $y_{1,j_1} + \dots + y_{m,j_m}$ (which is well defined by (i)) is non-zero.
\item[(iii)]  One has the cancellation
\begin{equation}\label{cancellation}
\sum_{j_1,\dots,j_m \in \{1,2\}} (-1)^{j_1+\dots+j_m} \operatorname{sgn}( y_{1,j_1} + \dots + y_{m,j_m} ) = 0.
\end{equation}
or equivalently (by (ii))
$$
\sum_{j_1,\dots,j_m \in \{1,2\}:  y_{1,j_1} + \dots + y_{m,j_m} > 0} (-1)^{j_1+\dots+j_m} = 0.
$$
That is to say, there are as many positive sums $y_{1,j_1} + \dots + y_{m,j_m} > 0$ with the index sum $j_1+\dots+j_m$ even as there are positive sums $y_{1,j_1} + \dots + y_{m,j_m} > 0$ with $j_1+\dots+j_m$ odd (and similarly with ``positive'' replaced by ``negative'').
\end{itemize}
Otherwise, we say that the pairs $\{y_{i,1}, y_{i,2}\}$ for $i=1,\dots,m$ have \emph{crossing sums}.
\end{definition}

\begin{remark}\label{inv} The notion of non-crossing sums is invariant with respect to interchanges between $y_{i,1}$ and $y_{i,2}$ for $i=1,\dots,m$, or between the pairs $\{y_{i,1},y_{i,2}\}$ for $i=1,\dots,m$, or by replacing each of the $y_{i,j}$ with their negations $-y_{i,j}$.  One could define this concept for other values of $m$ than $m=2,3$, but these are the only two values of $m$ we will need here.
\end{remark}

\begin{example}\label{ncs}  Let $a_1,a_2,b_1,b_2$ be distinct elements of $\R \cup \{-\infty\}$.  Then the pairs $\{a_1,a_2\}$ and $\{-b_1,-b_2\}$ have non-crossing sums if and only if the number of pairs $(i,j) \in \{1,2\} \times \{1,2\}$ with $a_i < b_j$ is even, thus for instance $\{a_1,a_2\}$ and $\{-b_1,-b_2\}$ will have non-crossing sums if
$$ a_1 < a_2 < b_1 < b_2$$
or
$$ a_1 < b_1 < b_2 < a_2$$
but not if
$$ a_1 < b_1 < a_2 < b_2.$$
In particular, if $-\infty < b_1 < b_2 < +\infty$, $\{ -\infty, a \}$ and $\{-b_1,-b_2\}$ have non-crossing sums if and only if $a$ lies outside of $[b_1,b_2]$.

In a more topological form: the pairs $\{a_1,a_2\}$ and $\{-b_1,-b_2\}$ have non-crossing sums iff it is possible to connect $(0,a_1)$ and $(0,a_2)$ (resp. $(0,b_1)$ and $(0,b_2)$) by a curve $\gamma_a$ (resp. $\gamma_b$) in the (one-point compactification of the) right half-plane $[0,+\infty) \times \R$, in such a manner that $\gamma_a$ and $\gamma_b$ do not cross.  See Figures \ref{cross-1}, \ref{cross-2}.
\end{example}

\begin{figure} [t]
\centering
\includegraphics[width=1.5in]{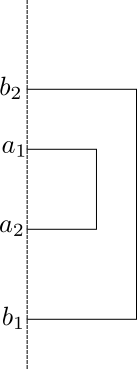}
\caption{The pairs $\{a_1, a_2\}$ and $\{-b_1,-b_2\}$ have non-crossing sums: the sums $a_1 + (- b_1), a_2 + (- b_1)$ are positive, while the sums
$a_1 + (-b_2), a_2 + (-b_2)$ are negative.  Note that the path connecting $(0,a_1)$ to $(0,a_2)$ does not cross the path connecting $(0,b_1)$ to $(0,b_2)$.}
\label{cross-1}
\end{figure}

\begin{figure} [t]
\centering
\includegraphics[width=1.5in]{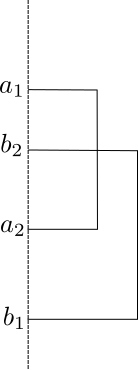}
\caption{The pairs $\{a_1, a_2\}$ and $\{-b_1,-b_2\}$ have crossing sums: the sums $a_1 + (- b_1), a_2 + (- b_1), a_1 + (-b_2)$ are positive, while the sum $a_2 + (-b_2)$ is negative.  Note that the path connecting $(0,a_1)$ to $(0,a_2)$ crosses the path connecting $(0,b_1)$ to $(0,b_2)$.}
\label{cross-2}
\end{figure}

\begin{example} The pairs $\{0,1\}, \{0,6\}, \{-5,1\}$ have non-crossing sums because there are as many positive sums $0 + 0 + 1, 0 + 6 - 5, 1 + 6 + 1$ with an even index sum as there are positive sums $1 + 0 + 1$, $1 + 6 - 5, 0 + 6 + 1$ with an odd index sum.  On the other hand, the pairs $\{-3,5\}, \{-3,5\}, \{-3,5\}$ have crossing sums because there are fewer positive sums $5+5+5$ with an even index sum than positive sums $-3+5+5, 5-3+5, 5+5-3$ with an odd index sum.

As we shall see later, the notion of $\{y_{i,1}, y_{i,2}\}$ for $i=1,2,3$ having non-crossing sums also has a topological interpretation (assuming axiom (i)), namely that there are curves $\gamma_i$ in the one-point compactification of $[0,+\infty) \times \R$ connecting $(0, y_{i,1})$ to $(0, y_{i,2})$ for $i=1,2,3$ such that there do not exist $x,y_1,y_2,y_3 \in \R$ with $y_1+y_2+y_3 = 0$ and $(x,y_i)$ in $\gamma_i$ for $i=1,2,3$.  This may help explain the terminology ``non-crossing sums''.
\end{example}

\begin{example}\label{diff}  Suppose that $0 < a < b < c$ and $x,y,z$ are real numbers such that $\{x,x+a\}, \{y,y+b\}, \{z,z+c\}$ have non-crossing sums.  The $2^3$ sums formed from this triplet may be almost completely ordered as
\begin{align*}
&x+y+z < x+a+y+z < x+y+b+z < x+a+y+b+z, x+y+z+c \\
&\quad < x+a+y+z+c < x+y+b+z+c < x+a+y+b+z+c
\end{align*}
(the reader may wish to first see this in the case $x=y=z=0$).  The sums $x+a+y+z, x+y+b+z, x+y+z+c, x+a+y+b+z+c$ have even index sum, and the other four sums have odd index sum.  Therefore, $\{x,x+a\}, \{y,y+b\}, \{z,z+c\}$ has non-crossing sums precisely when the origin $0$ falls in one of the five intervals
\begin{align*}
&(-\infty, x+y+z), \\
&(x+a+y+z, x+y+b+z), \\
&(x+a+y+b+z, x+y+z+c), \\
&(x+a+y+z+c, x+y+b+z+c), \\
&(x+a+y+b+z+c, +\infty)
\end{align*}
(with the third interval deleted if $x+y+z+c < x+a+y+b+z$).  In particular, we see that pair $\{x,x+a\}$ with the smallest difference has no influence on the sign of the triple sums, that is to say
$$ \operatorname{sgn}( x + y_2 + y_3 ) = \operatorname{sgn}( x+a + y_2 + y_3 )$$
for $y_2 \in \{y,y+b\}$ and $y_3 \in \{z,z+c\}$.  Conversely, if this pair has no influence on the sign of triple sums then the pairs $\{x,x+a\}, \{y,y+b\}, \{z,z+c\}$ have non-crossing sums.  This lack of influence by the pair with the smallest difference can thus be used as an alternate definition of non-crossing sums in the $m=3$ case (and it also works in the $m=2$ case).

One corollary of this analysis is that if $\{y,y+b\}$ has no influence on the sign of the triple sums, then neither does $\{x,x+a\}$; similarly, if $\{z,z+c\}$ has no influence on the non-crossing sums, then neither does $\{y,y+b\}$.
\end{example}

We are now ready to give the combinatorial formulation of Conjecture \ref{sai}.

\begin{conjecture}[Combinatorial formulation]\label{adf}  Let $k_1,k_2,k_3$ be odd natural numbers, and for each $i=1,2,3$, let $y_{i,1},\dots,y_{i,k_i}$ be distinct real numbers.  Adopt the convention that $y_{i,0} = y_{i,k_i+1} = -\infty$.  Assume the following axioms:
\begin{itemize}
\item[(i)]  (Non-crossing)  For any $1 \leq i \leq 3$ and $0 \leq p < q \leq k_i$ with $p,q$ the same parity, the pairs $\{ y_{i,p}, y_{i,p+1} \}$ and $\{ -y_{i,q}, -y_{i,q+1} \}$ have non-crossing sums.
\item[(ii)]  (Non-crossing sums)  For any $0 \leq p_1 \leq k_1$, $0 \leq p_2 \leq k_2$, $0 \leq p_3 \leq k_3$ with $p_1,p_2,p_3$ the same parity, the pairs $\{y_{1,p_1}, y_{1,p_1+1}\}$, $\{y_{2,p_2}, y_{2,p_2+1}\}$, $\{y_{3,p_3}, y_{3,p_3+1}\}$ have non-crossing sums.
\end{itemize}
Then one has the inequality
\begin{equation}\label{yio}
 \sum_{i=1}^3 \sum_{j=1}^{k_i} (-1)^{j-1} y_{i,j} < 0.
\end{equation}
\end{conjecture}

\begin{remark} In the language of Arn\'old, the hypothesis (i) shows that the ordering of the extended real numbers $-\infty, y_{i,1},\dots,y_{i,k_i}$ is given by the permutation of a \emph{meander} (formed by gluing together two non-crossing matchings); see \cite{meander}.
\end{remark}

The main result of this section is then

\begin{theorem}\label{aqv}  Conjecture \ref{sai} (and hence Conjecture \ref{area-1d}) is equivalent to Conjecture \ref{adf}.
\end{theorem}

\subsection{Forward direction}  Let us first assume Conjecture \ref{sai} and see how it implies Conjecture \ref{adf}.  Let $k_1,k_2,k_3$ and $y_{i,j}$ obey the assumptions of Conjecture \ref{adf}, but suppose for contradiction that \eqref{yio} failed.  The plan is then to use the quantities $y_{i,j}$ to build simple closed polygonal paths $\gamma_1,\gamma_2,\gamma_3 \colon \R/L\Z \to \mathtt{Cyl}_L$ for some $L>0$ to which Conjecture \ref{sai} may be applied.

By perturbing one of the $y_{i,j}$ slightly (noting that all the hypotheses on the $y_{i,j}$ are open conditions) we may assume that the quantity
\begin{equation}\label{Q-def}
 Q \coloneqq \sum_{i=1}^3 \sum_{j=1}^{k_i} (-1)^{j-1} y_{i,j} 
\end{equation}
is strictly positive.  Similarly, we may assume that the differences $|y_{i,p+1}-y_{i,p}|$ with $i=1,2,3$ and $1 \leq p < k_i$ are all distinct.

We will need a strictly monotone decreasing function $\phi \colon [0,+\infty) \to [1,2]$; the exact choice of $\phi$ is unimportant, but for concreteness one can take for instance $\phi(t) \coloneqq 1 + \frac{1}{1+t}$.

Let $L > 0$ be a sufficiently large quantity to be chosen later.  We will also need a certain large and negative quantity $-R$ (depending on $Q, L$, and the $y_{i,j}$) whose precise value will be specified later.

By applying Conjecture \ref{adf}(ii) with $p_1=p_2=p_3=0$, we see that $\{-\infty,y_{1,1}\}, \{-\infty,y_{2,1}\}, \{-\infty,y_{3,1}\}$ have non-crossing sums, which implies that
$$ y_{1,1} + y_{2,1} + y_{3,1} < 0.$$
Similarly if we apply Conjecture \ref{adf}(ii) with $p_1=k_1$, $p_2=k_2$, $p_3=k_3$, we see that
$$ y_{1,k_1} + y_{2,k_2} + y_{3,k_3} < 0.$$
As a consequence, we may find piecewise linear continuous functions $f_1,f_2,f_3 \colon [-1,1] \to \R$ such that
\begin{equation}\label{bound}
 f_i(-1) = y_{i,k_i}; \quad f_i(1) = y_{i,1} 
\end{equation}
for $i=1,2,3$, and such that
\begin{equation}\label{f123}
 f_1(t) + f_2(t) + f_3(t) < 0
\end{equation}
for all $-1 \leq t \leq 1$.  For instance, we can set $f_1,f_2,f_3$ to be the linear functions given by the boundary conditions \eqref{bound}.  But we can also subtract an arbitrary positive multiple of $1-|t|$ from any of $f_1,f_2,f_3$ and obey the above requirements.  In particular, there is a quantity $-C_0$ (independent of $L$) such that if the quantity $-R$ mentioned previously is less than or equal to $-C_0$, one can find $f_1,f_2,f_3$ solving the above conditions such that
\begin{equation}\label{fitt}
 \sum_{i=1}^3 \int_{-1}^1 f_i(t)\ dt = -R.
\end{equation}
Henceforth we fix $f_1,f_2,f_3$ with these properties.

Let $i=1,2,3$.  We construct some polygonal paths $\gamma_{i,1}, \gamma_{i,1 \to 2}, \dots, \gamma_{i,k_i-1 \to k_i}, \gamma_{i,k_i}, \gamma_{i,k_i \to 1}$ in $\mathtt{Cyl}_L$ by the following recipes:

\begin{definition}\label{recipe}\ 
\begin{itemize}
\item[(i)]  $\gamma_{i,1}$ is the rightward horizontal line segment from $(-\frac{L}{2} + 1, y_{i,1})$ to $(0, y_{i,1})$, projected to $\mathtt{Cyl}_L$.
\item[(ii)]  For any odd $1 \leq p < k_i$, $\gamma_{i,p \to p+1}$ is the piecewise linear path traversing the vertices
$$ (0, y_{i,p}), \left(\frac{L}{2} - \phi(|y_{i,p}-y_{i,p+1}|), y_{i,p}\right), \left(\frac{L}{2} - \phi(|y_{i,p}-y_{i,p+1}|), y_{i,p+1}\right), (0, y_{i,p+1})$$
in that order (that is to say, the concatenation of a rightward horizontal line segment, a vertical line segment, and a leftward horizontal line segment, if $L$ is large enough), and then projected to $\mathtt{Cyl}_L$.
\item[(iii)] For any even $1 < p < k_i$, $\gamma_{i,p \to p_1}$ is the piecewise linear path traversing the vertices
$$ (0, y_{i,p}), \left(-\frac{L}{2} + \phi(|y_{i,p}-y_{i,p+1}|), y_{i,p}\right), \left(-\frac{L}{2} + \phi(|y_{i,p}-y_{i,p+1}|), y_{i,p+1}\right), (0, y_{i,p+1})$$
in that order (that is to say, the concatenation of a leftward horizontal line segment, a vertical line segment, and a rightward horizontal line segment, if $L$ is large enough), and then projected to $\mathtt{Cyl}_L$.
\item[(iv)]  $\gamma_{i,k_i}$ is the rightward horizontal line segment from $(0, y_{i,k_i})$ to $(\frac{L}{2} - 1, y_{i,k_i})$, projected to $\mathtt{Cyl}_L$.
\item[(v)]  $\gamma_{i,k_i \to 1}$ is the graph
$$ \left\{ \left(\frac{L}{2} + t, f_i(t)\right): -1 \leq t \leq 1\right\}$$
traversed from left to right and then projected to $\mathtt{Cyl}_L$, thus it begins at $\pi_L(\frac{L}{2}-1, y_{i,k_i})$ and ends at $\pi_L(\frac{L}{2}+1, y_{i,1})$.
\end{itemize}
See Figure \ref{jam}.
\end{definition}

\begin{figure} [t]
\centering
\includegraphics[width=5in]{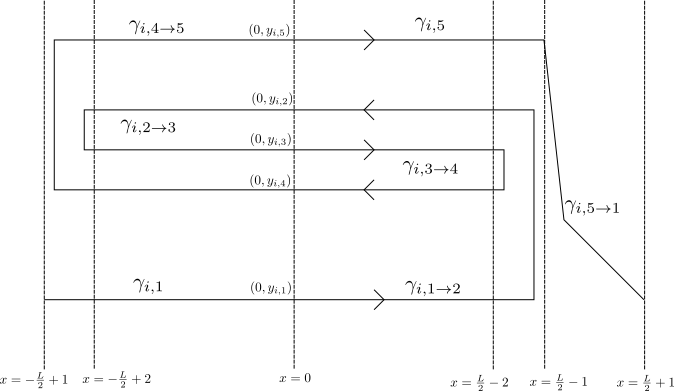}
\caption{Construction of the components of $\gamma_i$, in the case $k_i = 5$. Notice the interlacing between the $(0,y_{i,p})$ with $p$ odd and the $(0,y_{i,p})$ with $p$ even, and the alternating orientations of $\gamma_i$ at these locations.}
\label{jam}
\end{figure}

Clearly, one can concatenate the paths $\gamma_{i,1}, \gamma_{i,1 \to 2}, \dots, \gamma_{i,k_i-1 \to k_i}, \gamma_{i,k_i}, \gamma_{i,k_i \to 1}$ to form a closed polygonal path $\gamma_i$ in $\mathtt{Cyl}_L$ (which one can parameterise by $\R/L\Z$ after a suitable rescaling).  Using the convex fundamental domain $[-\frac{L}{2}+1,\frac{L}{2}+1] \times \R$ of $\mathtt{Cyl}_L$, we see that $\gamma_i$ is homotopic in this domain to the horizontal line segment from $(-\frac{L}{2}+1,y_{i,1})$ to $(\frac{L}{2}+1,y_{i,1})$, and hence $\gamma_i$ is homologous in $\mathtt{Cyl}_L$ to $\mathtt{Graph}_{0,L}$.

Using Conjecture \ref{adf}(i), we can show

\begin{lemma} Suppose $L$ is sufficiently large.  Then for any $i=1,\dots,3$, the path $\gamma_i$ is simple.
\end{lemma}

\begin{proof}  Each of the components $\gamma_{i,1}, \gamma_{i,1 \to 2}, \dots, \gamma_{i,k_i-1 \to k_i}, \gamma_{i,k_i}, \gamma_{i,k_i \to 1}$ of $\gamma_i$ are separately simple, and the endpoints are all distinct except for the endpoints of adjacent paths, so it suffices to show that no two of these components meet in the interior.

The interior of the path $\gamma_{i,k_i \to 1}$ lies in the strip $(\frac{L}{2}-1,\frac{L}{2}+1) \times \R$ (viewed as a subset of $\mathtt{Cyl}_L$), while the interior of the other paths lie in either $(-\frac{L}{2}+1, 0) \times \R$ or $(0, \frac{L}{2}-1) \times \R$, again only touching the boundary at endpoints.  So it suffices to show that there are no crossings in $(-\frac{L}{2}+1, 0) \times \R$ or $(0, \frac{L}{2}-1) \times \R$.

We just verify the claim for $(-\frac{L}{2}+1, 0) \times \R$, as the case of $(0, \frac{L}{2}-1) \times \R$ is completely analogous.  The only path components here are $\gamma_{i,1}$ and $\gamma_{i,p \to p+1}$ for $1 < p < k_i$ even.  To check that $\gamma_{i,1}$ and $\gamma_{i,p \to p+1}$ do not cross, it suffices from Definitions \ref{recipe}(i), (iii) to show that $y_{i,1}$ does not lie between $y_{i,p}$ and $y_{i,p+1}$.  But from Conjecture \ref{adf}(i) we see that $\{-\infty, y_{i,1}\}$ and $\{-y_{i,p}, -y_{i,p+1}\}$ have non-crossing sums, which gives the claim by Example \ref{ncs}.

Now we need to check that $\gamma_{i,p \to p+1}$ and $\gamma_{i, q \to q+1}$ do not cross when $1 < p, q < k_1$ are even and distinct.  By Conjecture \ref{adf}(i), the pairs $\{ y_{i,p}, y_{i,p+1} \}$ and $\{ -y_{i,q}, -y_{i,q+1} \}$ have non-crossing sums; thus the interval spanned by $\{ y_{i,q}, y_{i,q+1} \}$ either is disjoint from, contains, or is contained in the interval spanned by $\{ y_{i,p}, y_{i,p+1}\}$. In the former case, the paths $\gamma_{i,p \to p+1}$ and $\gamma_{i, q \to q+1}$ are clearly disjoint because from Definition \ref{recipe}(ii), the $y$ coordinate of any point on the first path lies in the interval spanned by $\{ y_{i,p}, y_{i,p+1}\}$, and the $y$ coordinate on any point on the second path lies in the interval spanned by 
$\{ y_{i,q}, y_{i,q+1}\}$.  By symmetry, the only remaining case to check is when the interval spanned by $\{ y_{i,p}, y_{i,p+1}\}$ is contained in the interval spanned by $\{ y_{i,q}, y_{i,q+1}\}$.  But in this case, we have $|y_{i,p}-y_{i,p+1}| < |y_{i,q} -y_{i,q+1}|$, so by Definition \ref{recipe}(iii) and the monotone decreasing nature of $\phi$, the vertical segment of the curve $\gamma_{i,p \to p+1}$ lies to the right of that of $\gamma_{i,q \to q+1}$.  From this we see that the two curves are disjoint.  This concludes the demonstration of simplicity in $(-\frac{L}{2}+1, 0) \times \R$; the case of $(0,\frac{L}{2}-1) \times \R$ is similar.
\end{proof}

In a similar fashion, we can use Conjecture \ref{adf}(ii) to show

\begin{lemma}  There does not exist $x \in \R/L\Z$ and $y_1,y_2,y_3 \in \R$ with $y_1+y_2+y_3 = 0$ such that $(x,y_i) \in \gamma_i(\R)$ for all $i=1,2,3$.
\end{lemma}

\begin{proof} Suppose for contradiction that $x,y_1,y_2,y_3$ exist with the stated properties. 

First suppose that $x$ lies in $[\frac{L}{2}-1, \frac{L}{2}+1]$ (projected to $\R/L\Z$).  Then from Definition \ref{recipe}(v) we have $y_i = f_i(x - \frac{L}{2})$ for $i=1,2,3$, but then from \eqref{f123} we cannot have $y_1+y_2+y_3=0$, a contradiction.

We now treat the case when $x$ lies in $[0, \frac{L}{2}-1]$ (projected to $\R/L\Z$); the remaining case when $x$ lies in $[-\frac{L}{2}+1,0]$ is similar and will be omitted.  By Definition \ref{recipe}, we see that each of the $(x,y_i)$ lies either on $\gamma_{i,k_i}$ or on $\gamma_{i, p_i \to p_i+1}$ for some odd $1 \leq p_i < k_i$. 

Suppose first that each of the $(x,y_i)$ lie on $\gamma_{i,p_i \to p_i+1}$.  By hypothesis, the quantities $\frac{L}{2} - \phi(|y_{i,p_i} - y_{i,p_i+1}|)$ for $i=1,2,3$ are distinct; by cyclic permutation we may assume that $\frac{L}{2} - \phi(|y_{1,p_1} - y_{1,p_1+1}|)$ is the smallest of these quantities, or equivalently that $|y_{i,p_i}-y_{i,p_i+1}|$ is minimised at $i=1$.  By Definition \ref{recipe}(ii), the $x$ coordinate of $\gamma_{1,p_1 \to p_1+1}$ does not exceed $\frac{L}{2} - \phi(|y_{1,p_1} - y_{1,p_1+1}|)$, which implies that
$$ x \leq \frac{L}{2} - \phi(|y_{1,p_1} - y_{1,p_1+1}|),$$
and hence by further application of Definition \ref{recipe}(ii) we have
$y_2 = y_{2,q_2}$ and $y_3 = y_{3,q_3}$ for some $q_2 \in \{p_2,p_2+1\}$ and $q_2 \in \{p_3,p_3+1\}$; furthermore $y_1$ lies between $y_{1,p_1}$ and $y_{1,p_1+1}$ inclusive.  Since $y_1+y_2+y_3=0$, this implies that the sums $y_{1,p_1} + y_{2,q_2} + y_{3,q_3}$ and $y_{1,p_1+1} + y_{2,q_2} + y_{3,q_3}$ do not have the same sign.  Because $i=1$ minimises $|y_{i,p_i}-y_{i,p_i+1}|$, we conclude (from Example \ref{diff}) that the pairs $\{ y_{1,p_1}, y_{1,p_1+1}\}$, $\{ y_{2,p_2}, y_{2,p_2+1}\}$, $\{ y_{3,p_3}, y_{3,p_3+1}\}$ do not have non-crossing sums, contradicting Conjecture \ref{adf}(ii).

The case when one or more of the $(x,y_i)$ lies on $\gamma_{i,k_i}$ is treated similarly, with $k_i$ now playing the role of $p_i$ (and recalling the convention $y_{i,k_i+1} = -\infty$.  This concludes the treatment of the case $x \in [0, \frac{L}{2}-1]$, and the case $x \in [-\frac{L}{2}+1,0]$ is similar.
\end{proof}

From the previous two lemmas and Conjecture \ref{sai}, we conclude that
\begin{equation}\label{doop}
 \int_{\gamma_1+\gamma_2+\gamma_3} y\ dx \neq 0
\end{equation}
We work in the fundamental domain $[-\frac{L}{2}+1, \frac{L}{2}+1] \times \R$ of $\mathtt{Cyl}_L$.  On the strip $[\frac{L}{2}-1, \frac{L}{2}+1] \times \R$, the contribution to $\int_{\gamma_1+\gamma_2+\gamma_3} y\ dx$ is $\sum_{i=1}^3 \int_{-1}^1 f_i(t)\ dt = -R$ thanks to Definition \ref{recipe}(v) and \eqref{fitt}.  On the strip $[-\frac{L}{2}+2, \frac{L}{2}-2] \times \R$, the curve $\gamma_i$ for $i=1,2,3$ is simply the union of the line segments $[-\frac{L}{2}+2, \frac{L}{2}-2] \times \{y_{i,p}\}$ for $p=1,\dots,k_i$ (traversed from left to right for odd $p$, and right to left for even $p$), so the contribution to $\int_{\gamma_1+\gamma_2+\gamma_3} y\ dx$ here is
$$ \sum_{i=1}^3 (L-4) \sum_{j=1}^{k_i} (-1)^{j-1} y_{i,j} = (L-4) Q$$
thanks to \eqref{Q-def}.  Finally, the contribution of the remaining strips $[-\frac{L}{2}+1, -\frac{L}{2}+2] \times \R \cup [\frac{L}{2}-2, -\frac{L}{2}-1] \times \R$ is some quantity $-C_1$ independent of $L$ and $R$, as can be seen by translating these strips by $\pm \frac{L}{2}$.  The inequality \eqref{doop} thus becomes
$$ (L-4) Q - C_1 - R \neq 0.$$
But as $Q$ is positive, we can make this quantity vanish by choosing $L$ large enough and then setting $-R \coloneqq C_1 - (L-4)Q$; note for $L$ large enough that this value of $-R$ will be less than the threshold $-C_0$ needed so that one can arrange the function $f_1,f_2,f_3$ to obey \eqref{fitt}.  This yields the desired contradiction.

\subsection{Backward direction}  Now we assume Conjecture \ref{adf} and see how it implies Conjecture \ref{sai}.

By applying (a slight variant of) Proposition \ref{lim2}, we see that to prove Conjecture \ref{sai}, it suffices to do so under the additional nondegeneracy hypothesis that all the vertices of $\gamma_1, \gamma_2, \gamma_3$ have distinct $x$-coordinates (in particular, these curves do not contain any vertical edges).  Write $\gamma_i(t) = (x_i(t),y_i(t))$ for some piecewise linear $x_i \colon \R/L\Z \to \R/L\Z$ and $y_i \colon \R/L\Z \to \R$.  As $\gamma_i$ is homologous to $\mathtt{Graph}_{0,L}$, we can find a continuous lift $\tilde x_i \colon \R \to \R$ of $x_i$ with $\tilde x_i(t+L) = \tilde x_i(t)+L$ for all $t \in \R$; we also let $\tilde y_i \colon \R \to \R$ be the periodic lift of $y_i$.  As $\tilde x_i(t) - t$ is periodic and continuous, it is bounded; by multiplying the period $L$ by a large integer if necessary, we may assume that
\begin{equation}\label{x-stable}
|\tilde x_i(t)-t| \leq \frac{L}{10}
\end{equation}
for all $t \in \R$ and $i=1,2,3$.

Using the nondegeneracy hypothesis, we see that for any $i=1,2,3$ and any $x \in \R$, the fibre $\{ y \in \R: (x,y) \in \gamma_i(\R) \}$ consists of a finite number $k_i(x)$ of real numbers, where the function $k_i$ takes values in the odd natural numbers, is periodic in $L$, and is locally constant for all $x$ outside of finitely many residue classes mod $L$.  We can enumerate these real numbers as
$$ \tilde y_i( t_{i,1}(x) ), \dots, \tilde y_i( t_{i,k_i(x)}(x) )$$
where $t_{i,1}(x) < \dots < t_{i,k_i(x)}(x)$ are those real numbers $t$ with $\tilde x_i(t) = x$, arranged in increasing order, thus
\begin{equation}\label{yan}
\tilde x_i(t_{i,j}(x)) = x
\end{equation}
for all $x \in \R$, $i=1,2,3$, and $1 \leq j \leq k_i(x)$.  For $x$ outside of finitely many residue classes mod $L$, the functions $t_{i,j}$ are locally linear, and have the $L\Z$-equivariance property
$$ t_{i,j}(x+L) = t_{i,j}(x) + L$$
for all $x \in \R$, $i=1,2,3$, and $1 \leq j \leq k_i(x)$.  Also, from the nondegeneracy hypothesis we see from the intermediate value theorem that outside of finitely many residue classes mod $L$, $t_{i,j}$ is increasing for odd $j$ and decreasing for even $j$.

The analogue of the set $\overline{\mathtt{Squares}_L} \subset \mathtt{Cyl}_L^4$ in this context is the oriented $3$-dimensional submanifold $\mathtt{Sums}$ of $\mathtt{Cyl}_L^3$ defined by
$$ \mathtt{Sums} \coloneqq \{ ((x_1,y_1),(x_2,y_2),(x_3,y_3)) \in \mathtt{Cyl}_L^3: x_1=x_2=x_3; y_1+y_2+y_3 = 0\}.$$
The hypotheses of Conjecture \ref{sai} then assert that the $3$-cycle $\gamma_1 \times \gamma_2 \times \gamma_3$ in $\mathtt{Cyl}_L^3$ does not intersect $\mathtt{Sums}$.  We view $\mathtt{Sums}$ as an oriented submanifold of the $4$-manifold $V \subset \mathtt{Cyl}_L^3$ defined by
$$ V \coloneqq \{ ((x_1,y_1),(x_2,y_2),(x_3,y_3)) \in \mathtt{Cyl}_L^3: x_1=x_2=x_3\}.$$
As we assumed the vertices of $\gamma_1,\gamma_2,\gamma_3$ to have distinct $x$ coordinates, the $3$-cycle $\gamma_1 \times \gamma_2 \times \gamma_3$ intersects $V$ transversely in some $1$-cycle $\gamma_{123}$.  As $\gamma_1,\gamma_2,\gamma_3$ are homologous to $\mathtt{Graph}_{0,L}$, $\gamma_{123}$ is homologous to the $1$-cycle
$$ \mathtt{Graph}_{0,L}^\Delta \coloneqq \{ (p,p,p): p \in \mathtt{Graph}_{0,L} \}$$
with the standard orientation.  

Now we argue as in the previous section.  The $4$-manifold $V$ is homeomorphic to $\R/L\Z \times \R^3$ and thus has first homology generated by $\mathtt{Graph}_{0,L}^\Delta$.  By the greedy algorithm, one can express the $1$-cycle $\gamma_{123}$ as a finite integer linear combination of closed paths contained (as a set) in $\gamma_{123}$, each of which is either simple or a $1$-boundary; one of these, say $\gamma_{123}^0 \colon \R/L\Z \to V$, is homologous to $m \mathtt{Graph}_{0,L}^\Delta$ for some non-zero integer $m$, and is thus simple.  From the Jordan curve theorem, $m$ must be $+1$ or $-1$; by reversing the orientation of $\gamma_{123}^0$ we can then assume that $m=1$, thus $\gamma_{123}^0$ is homologous to $\mathtt{Graph}_{0,L}^\Delta$ and is contained (as a set) in $\gamma_{123}$.  In particular, it avoids $\mathtt{Sums}$.  If we write
$$ \gamma_{123}^0(t) = ((X(t),Y_1(t)), (X(t),Y_2(t)), (X(t),Y_3(t)))$$
then $X \colon \R/L\Z \to \R/L\Z$, $Y_1,Y_2,Y_3 \colon \R/L\Z \to \R$ are piecewise linear continuous functions with $X$ homologous to the identity function (in the sense that $X$ lifts to a function $\tilde X \colon \R \to \R$ with $\tilde X(t+L) = \tilde X(t)+L$ for all $t \in \R$), with the properties that
\begin{equation}\label{dap}
(X(t),Y_i(t)) \in \gamma_i(\R/L\Z)
\end{equation}
and
\begin{equation}\label{ysum}
 Y_1(t)+Y_2(t)+Y_3(t) \neq 0
\end{equation}
for all $t \in \R/L\Z$ and $i=1,2,3$.  

\begin{remark}\label{dyn}
One can view the functions $X(t), Y_1(t), Y_2(t), Y_3(t)$ from a dynamical perspective by thinking of $(X(t),Y_i(t))$ as the trajectory of a particle $P_i$ that is constrained to lie in $\gamma_i(\R/L\Z)$, and with all three particles $P_1,P_2,P_3$ constrained to lie on a vertical line.  We can also constrain the particles $P_1,P_2,P_3$ to have a constant horizontal speed; the particles move in one horizontal direction until one of the particles $P_i$ cannot move any further due to it hitting a vertex $v$ of $\gamma_i$ with both edges adjacent to $v$ lying on the same side of the vertical line containing $v$.  Whenever such a collision occurs, the horizontal velocity reverses sign, $P_i$ moves from one edge of $\gamma_i$ to the next, and the other two particles reverse themselves and retrace their steps; see Figure \ref{collide}.  Note from our hypotheses that only one collision occurs at a time.  Because the paths $\gamma_1,\gamma_2,\gamma_3$ have only finitely many edges, these trajectories must be periodic; the above homological considerations ensure that at least one of these trajectories is homologous to $\mathtt{Graph}_{0,L}^\Delta$ (possibly after enlarging the period $L$).
\end{remark}

\begin{figure} [t]
\centering
\includegraphics[width=3.5in]{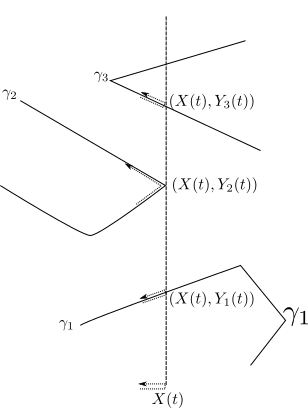}
\caption{The dynamics of $X(t), Y_1(t), Y_2(t), Y_3(t)$.}
\label{collide}
\end{figure}

Recall that $X \colon \R/L\Z \to \R/L\Z$ lifts to a function $\tilde X \colon \R \to \R$ such that $\tilde X(t+L) = \tilde X(t)+L$ for all $t \in \R$, thus $\tilde X(t) - t$ is periodic and therefore bounded.  We also lift $Y_i: \R/L\Z \to \R$ periodically to $\tilde Y_i: \R \to \R$ for $i=1,2,3$.  By lifting and \eqref{dap}, we can then find unique continuous functions $T_i \colon \R \to \R$ with $T_i(t+L) = T_i(t)+L$ for all $t \in \R$, such that
\begin{equation}\label{xan}
 (\tilde X(t), \tilde Y_i(t)) = (\tilde x_i(T_i(t)), \tilde y_i(T_i(t)))
\end{equation}
for all $t \in \R$.  By replacing $L$ with a large multiple if necessary, we may assume that
\begin{equation}\label{xat}
 |\tilde X(t) - t|, |T_i(t)-t| \leq \frac{L}{10}
\end{equation}
for all $t \in \R$ and $i=1,2,3$.  

By continuity and \eqref{ysum}, we see that the expression $\tilde Y_1(t) + \tilde Y_2(t) + \tilde Y_3(t)$ is either positive for all $t$, or negative for all $t$.  By applying the reflection $(x,y) \mapsto (x,-y)$ on $\mathtt{Cyl}_L$ we may assume the latter case occurs, thus
\begin{equation}\label{yot}
 \tilde Y_1(t)+\tilde Y_2(t)+\tilde Y_3(t) < 0
\end{equation}
for all $t \in \R/L\Z$.  To establish Conjecture \ref{sai}, it suffices to show that
$$ \int_{\gamma_1+\gamma_2+\gamma_3} y\ dx < 0;$$
integrating fibre by fibre, it will suffice to show that
\begin{equation}\label{flip}
 \sum_{i=1}^3 \sum_{j=1}^{k_i(x)} (-1)^{j-1} \tilde y_i(t_{i,j}(x)) < 0 
\end{equation}
for almost every $x \in \R$.

Fix $x \in \R$; we may assume that $x$ avoids all the $x$ coordinates of vertices of $\gamma_1,\gamma_2,\gamma_3$.  We abbreviate $k_i = k_i(x)$ and
$$ y_{i,j} \coloneqq \tilde y_i(t_{i,j}) $$
for $i=1,2,3$ and $1 \leq j \leq k_i$, adopting the conventions $y_{i,0} = y_{i,k_i+1} = -\infty$.
Applying Conjecture \ref{adf}, it will suffice to verify the hypotheses (i), (ii) of that conjecture.

Let $t_+ = t_+(x)$ denote the largest time $t_+ \in \R$ for which $\tilde X(t_+) = x$ (such a time exists thanks to \eqref{xat} and continuity).  We claim that $T_i(t_+) = t_{i,k_i}(x)$ for all $i=1,2,3$.  Suppose for contradiction that this failed for some $i=1,2,3$, then from \eqref{dap} one has $T_i(t_+) = t_{i,p}(x)$ for some $1 \leq p < k_i$.  From \eqref{xat} and the intermediate value theorem we must then have $T_i(t) = t_{i,k_i(x)}(x)$ for some $t>t_+$, which by \eqref{xan}, \eqref{yan} gives $\tilde X(t) = x$, contradicting the maximality of $t_+$.  Similarly, if $t_- = t_-(x)$ is the smallest time $t_- \in \R$ for which $\tilde X(t_-) = x$, then $T_i(t_-) = t_{i,1}(x)$ for $i=1,2,3$.  
From \eqref{yot} applied at the times $t_+, t_-$ we have the inequalities
\begin{equation}\label{ineq-1}
y_{1,1} + y_{2,1} + y_{3,1} < 0
\end{equation}
and
\begin{equation}\label{ineq-2}
y_{1,k_1} + y_{2,k_2} + y_{3,k_3} < 0.
\end{equation}
Having obtained these inequalities, we will have no further need of the functions $X, \tilde X, Y_i, \tilde Y_i, T_i$ or the curves $\gamma_{123}^0$, although we will introduce a variant of these functions shortly.

We now verify the non-crossing property (i) for any given $i=1,2,3$.  We just verify the claim when $p,q$ are odd, as the claim when $p,q$ are even is completely analogous.  First suppose that $q < k_i$.  Let $\gamma_{x,i,p}$ denote the restriction of path $t \mapsto (\tilde x_i( t ), \tilde y_i( t ))$ to the interval $t \in [t_{i,p}(x), t_{i,p+1}(x)]$; define $\gamma_{x,i,q}$ similarly.  The path $\gamma_{x,i,p}$ traces out a piecewise linear curve in $\R^2$ that starts at $(x, y_{i,p})$, ends at $(x, y_{i,p+1})$, and does not encounter the vertical line $x \times \R$ at any point in between; also, it moves to the right for $t$ near $t_{i,p}(x)$ (and to the left for $t$ near $t_{i,p+1}(x)$).  Thus, this curve $\gamma_{x,i,p}$ must stay in the right half-plane $[x,+\infty) \times \R$; actually, by \eqref{x-stable} we see that it stays in the strip $[x, x+\frac{L}{2}] \times \R$ (say).  Similarly for $\gamma_{x,i,q}$.  On the other hand, as $\gamma_i$ is simple and $p<q$, the two paths $\gamma_{x,i,p}$ and $\gamma_{x,i,q}$ cannot meet.  From the Jordan curve theorem and Example \ref{ncs}, this forces the endpoints $\{ y_{i,p}, y_{i,p+1} \}$ and $\{ y_{i,q}, y_{i,q+1} \}$ of these paths to be non-crossing, giving (i) in this case.

Now suppose that $q = k_i$.  In this case we define $\gamma_{x,i,q}$ to be the restriction of $t \mapsto (\tilde x_i( t ), \tilde y_i( t ))$ to the interval $[t_{i,k_i(x)}(x), t_{i,1}(x) + L]$ (this interval is well-defined by \eqref{x-stable}).  This is a path from $(x, y_{i,k_i})$ to $(x+L, y_{i,1})$ that does not cross $\{x,x+L\} \times \R$ except at endpoints, and hence lies in the strip $[x,x+L] \times \R$.  It cannot cross $\gamma_{x,i,p}$, which also lies in this strip, avoids the right edge, and starts and ends at the points $(x,y_{i,p}), (x,y_{i,p+1})$ on the left edge.  By the Jordan curve theorem, this implies that $y_{i,k_i}$ cannot lie between $y_{i,p}$ and $y_{i,p+1}$, which by Example \ref{ncs} implies (since $y_{i,k_i+1}=-\infty$) that $\{ y_{i,p}, y_{i,p+1}\}$ and $\{y_{i,k_i}, y_{i,k_i+1}\}$ are non-crossing.  This concludes the establishment of (i) when $p,q$ are odd; the case when $p,q$ are even is analogous (working to the left of $\{x\} \times \R$ rather than to the right, and using the convention $y_{i,0} = -\infty$ rather than $y_{i,k_i+1} = -\infty$) and is omitted.

Now we verify (ii) for $0 \leq p_1 \leq k_1$, $0 \leq p_2 \leq k_2$, $0 \leq p_3 \leq k_3$ with $p_1,p_2,p_3$ the same parity.  We just establish the claim when $p_1,p_2,p_3$ are all odd, as the case when $p_1,p_2,p_3$ are all even is completely analogous.

In the case $p_1 = k_1, p_2 = k_2, p_3 = k_3$, we see from \eqref{ineq-2} and the conventions $y_{1,k_1+1} = y_{2,k_2+2} = y_{3,k_3+3}=-\infty$ that the pairs $\{ y_{1,k_1}, y_{1,k_1+1}\}$, $\{ y_{2,k_2}, y_{2,k_2+1}\}$, $\{ y_{3,k_3}, y_{3,k_3+1}\}$.  Thus we may assume that $p_i < k_i$ for at least one $i$; say $p_1 < k_1$.

As in the proof of (i), we can form the curves $\gamma_{x,i,p_i}$ for $i=1,2,3$, which lie in the strip $[x, x+L] \times \R$, with initial point $(x, y_{i,p_i})$ and final point $(x, y_{i,p_{i+1}})$.  For $i=1,2,3$, let $x_i$ denote the maximum $x$ coordinate attained by $\gamma_{x,i,p_i}$, thus $x < x_i \leq x+L$; furthermore $x_i = x+L$ if $p_i = k_i$ and $x_i \leq x + \frac{L}{2}$ otherwise, in particular $x_1 \leq x + \frac{L}{2}$.  As the vertices of $\gamma_i$ have distinct $x$ coordinates, the $x_i$ are distinct in the interval $[x,x+\frac{L}{2}]$; without loss of generality we may then take $x_1 < x_2,x_3$.  For $i=1,2,3$, define $\gamma'_{x,i,p_i}$ to be the connected component of $\gamma_{x,i,p_i} \cap [x, x_1] \times \R$ that contains $(x,y_{i,p_i})$; thus $\gamma'_{x,1,p_1} = \gamma_{x,1,p_1}$, and for $i=2,3$, $\gamma'_{x,i,p_i}$ is a piecewise path connecting $(x,y_{i,p_i})$ to some point on the vertical line $\{x_1\} \times \R$.

Consider the set
$$ S \coloneqq \{ (x',y'_1,y'_2,y'_3) \in \R^4: (x',y'_i) \hbox{ lies in } \gamma'_{x,i,p_i} \hbox{ for } i=1,2,3 \}.$$
The set $S$ is a union of line segments in $\R^4$.  It contains the point $(x, y_{1,p_1}, y_{2,p_2}, y_{3,p_3})$ with exactly one line segment emenating from it; $S$ similarly contains the point $(x, y_{1,p_1+1}, y_{2,p_2}, y_{3,p_3})$ with exactly one line segment emenating from it.  A local analysis (using the non-degeneracy hypothesis that the vertices of $\gamma_1,\gamma_2,\gamma_3$ all have distinct $x$ coordinates) then reveals that every other point $(x',y'_1,y'_2,y'_3)$ in $S$ is either an interior point of a line segment in $S$ (and avoids all other line segments comprising $S$), or else is a vertex that is the endpoint of exactly two edges in $S$; this claim is most delicate in the case where $x' = x_1$, in which the curves $\gamma'_{x,2,p_2}$ and $\gamma'_{x,3,p_3}$ have terminated, but $\gamma_{x,1,p_1}$ leaves $(x',y_1)$ in two leftward directions, thus again forming two edges in $S$ (see Figure \ref{turning}).  Because of this, there must be a path $t \mapsto (X'(t),Y'_1(t), Y'_2(t), Y'_3(t))$ in $S$ from $(x, y_{1,p_1}, y_{2,p_2}, y_{3,p_3})$ to $(x, y_{1,p_1+1}, y_{2,p_2}, y_{3,p_3})$ (cf. Remark \ref{dyn}).  By the hypothesis of Conjecture \ref{sai}, we must have
$$ Y'_1(t) + Y'_2(t) + Y'_3(t) \neq 0$$
for all $t$.  In particular, we conclude that the sums
$$ y_{1,p_1} + y_{2,p_2} + y_{3,p_3}, y_{1,p_1+1} + y_{2,p_2} + y_{3,p_3}$$
have the same sign.  A similar argument (using the connected component of $\gamma'_{x,i,p_i}$ containing $(x,y_{i,p_i+1})$ rather than $(x,y_{i,p_i})$ as appropriate) shows more generally that the sums
$$ y_{1,p_1} + y_{2,q_2} + y_{3,q_3}, y_{1,p_1+1} + y_{2,q_2} + y_{3,q_3}$$
have the same sign for $q_2 \in \{p_2,p_2+1\}$ and $q_3 \in \{p_3,p_3+1\}$ (this claim is trivially true when $q_2 = k_2+1$ or $q_3=k_3+1$.  By Definition \ref{ncs-def}, we conclude that the pairs $\{y_{1,p_1}, y_{1,p_1+1}\}$, $\{y_{2,p_2}, y_{2,p_2+1}\}$, $\{y_{3,p_3}, y_{3,p_3+1}\}$ have non-crossing sums, giving (ii) in the case that $p_1,p_2,p_3$ are all odd; the claim when $p_1,p_2,p_3$ are all even are proven similarly (using the convention $y_{i,0}=-\infty$ instead of $y_{i,k_i}=-\infty$, working to the left of $\{x\} \times \R$ rather than to the right, and using \eqref{ineq-1} in place of \eqref{ineq-2}) and is omitted.  This completes the derivation of Conjecture \ref{sai} from Conjecture \ref{adf}, and establishes Theorem \ref{aqv}.

\begin{figure} [t]
\centering
\includegraphics[width=3in]{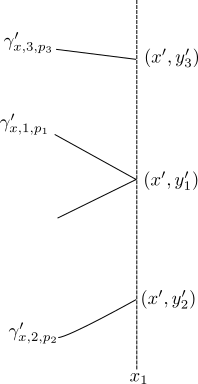}
\caption{The local behaviour of a point $(x',y'_1,y'_2,y'_3)$ in $S$ when $x'$ equals $x_1$.}
\label{turning}
\end{figure}

\section{Some special cases of Conjecture \ref{adf}}

From Theorem \ref{aqv}, we see that any counterexample to Conjecture \ref{adf} can be converted to counterexamples for Conjectures \ref{sai}, \ref{area-1d}, and \ref{area-ineq}, and we believe it likely that such counterexamples, should they exist, could then be modified to give counterexamples to Conjectures \ref{qpspp}, \ref{pspp}, \ref{pspps}, and hence the original square peg conjecture (Conjecture \ref{spp}).  On the other hand, after extensive testing of examples, the author is now inclined to believe that Conjecture \ref{adf} is true, and a proof of this conjecture is likely to lead to an approach to establish Conjecture \ref{area-ineq} (and hence Conjectures \ref{qpspp}, \ref{pspp}, \ref{pspps}) and perhaps even Conjecture \ref{spp}.

We do not have a proof of Conjecture \ref{adf} in full generality, however we can verify some special cases.  Firstly, we observe the following analogue of Theorem \ref{ook} for Conjecture \ref{sai}:

\begin{theorem}\label{koo}  Conjecture \ref{sai} is true when one of the curves $\gamma_i$ is the graph $\gamma_i = \mathtt{Graph}_f$ of a piecewise linear function $f: \Z/L\Z \to R$.
\end{theorem}

\begin{proof}  Suppose that $\gamma_3 = \mathtt{Graph}_f$.  By replacing $\gamma_3$ with $\mathtt{Graph}_{0,L}$ and $\gamma_2$ with the transformed 
polygonal path $\{ (x,y-f(x)): (x,y) \in \gamma_2(\Z/L\Z)\}$, we may assume without loss of generality that $f=0$.  The hypothesis of Conjecture \ref{sai} then ensures that the reflection $\tilde \gamma_2$ of $\gamma_2$ across the $x$ axis is disjoint from $\gamma_1$, hence by the Jordan curve theorem and Stokes' theorem as before we have
$$ \int_{\tilde \gamma_2} y\ dx \neq \int_{\gamma_1} y\ dx$$
giving Conjecture \ref{sai} in this case since $\int_{\tilde \gamma_2} y\ dx = -\int_{\gamma_2} y\ dx$.
\end{proof}

From this theorem and the construction used in the proof of Theorem \ref{aqv}, we see that Conjecture \ref{adf} holds when one of the $k_i$, say $k_3$, is equal to $1$.
Of course, as Conjecture \ref{adf} is largely a combinatorial conjecture, one expects to also be able to verify the $k_3=1$ case of Conjecture \ref{adf} by a direct combinatorial argument, without explicit invocation of the Jordan curve theorem.  We can do this by developing some combinatorial analogues of the Jordan curve theorem that are valid even when $k_1,k_2,k_3 > 1$.  For any $i \in \{1,2,3\}$ and $y \in [-\infty,+\infty]$, define the \emph{winding number} $W_i(y)$ by the Alexander numbering rule \cite{alex}
\begin{equation}\label{Wi-def}
\begin{split}
W_i(y) &\coloneqq \frac{1}{2} + \frac{1}{2} \sum_{j=1}^{k_i} (-1)^{j-1} \operatorname{sgn}( y_{i,j} - y ) \\
&= \frac{1}{2} \sum_{j=1}^{k_i} (-1)^{j-1} (1 + \operatorname{sgn}( y_{i,j} - y ))
\end{split}
\end{equation}
(where we use the hypothesis that $k_i$ is odd), thus $W_i$ is a half-integer on $y_{i,1},\dots,y_{i,k_i}$, a locally constant integer outside of these points, and jumps by $\pm \frac{1}{2}$ when one perturbs off of one of the $y_{i,j}$ in either direction.  Also observe that $W_i(y)$ equals $0$ near $+\infty$, and $1$ near $-\infty$.  From Fubini's theorem we can relate the winding number to the alternating sum $\sum_{j=1}^{k_i} (-1)^{j-1} y_{i,j}$ by the identity
\begin{equation}\label{fubini}
\int_{-T}^\infty W_i(y)\ dy = \sum_{j=1}^{k_i} (-1)^{j-1} y_{i,j} + T
\end{equation}
which holds for all sufficiently large $T$.  A similar argument gives
\begin{equation}\label{fubini-sym}
\int_{-T}^\infty (1-W_i(-y))\ dy = -\sum_{j=1}^{k_i} (-1)^{j-1} y_{i,j} + T
\end{equation}
again for sufficiently large $T$.

We can then use the hypothesis (i) of Conjecture \ref{adf} to give

\begin{lemma}[Combinatorial Jordan curve theorem]\label{cjct}  Suppose that the hypotheses of Conjecture \ref{adf} hold.  Let $i=1,2,3$.  Then one has $W_i(y_{i,j}) = \frac{1}{2}$ for all $j=1,\dots,k_i$, and $W_i(y) \in \{0,1\}$ for all $y$ in $[-\infty,+\infty] \backslash \{y_{i,1},\dots,y_{i,k+1}\}$.
\end{lemma}

\begin{proof}  Because $W_i$ is locally constant away from $\{y_{i,1},\dots,y_{i,k+1}\}$ and jumps by $\pm \frac{1}{2}$ when it reaches any of the $y_{i,j}$, it suffices to establish the first claim.  Let $1 \leq p < k_i$.  From Conjecture \ref{adf}(i) we have
$$ \sum_{j = q, q+1} (-1)^{j-1} \operatorname{sgn}( y_{i,p} - y_{i,j} ) = \sum_{j = q, q+1} (-1)^{j-1} \operatorname{sgn}( y_{i,p+1} - y_{i,j} ) $$
for all $0 \leq q \leq k_1$ distinct from $p$ with the same parity as $p$ (using the conventions $y_{i,0}=y_{i,k_1+1} = -\infty$).  Direct inspection shows that the claim also holds for $q=p$.  Summing over $q$, and noting that the contributions of $j=0$ or $j=k_1+1$ are the same on both sides, we conclude that
$$ \sum_{j = 1}^{k_1} (-1)^{j-1} \operatorname{sgn}( y_{i,p} - y_{i,j} ) = \sum_{j = 1}^{k_1} (-1)^{j-1} \operatorname{sgn}( y_{i,p+1} - y_{i,j} ) $$
and hence
$$ W_i(y_{i,p}) = W_i(y_{i,p+1})$$
for all $1 \leq p < k_i$.  Direct computation also shows that $W_i(y_{i,p})=\frac{1}{2}$ when $1 \leq p \leq k_1$ maximises $y_{i,p}$, and the claim follows.
\end{proof}

Next, for distinct $i,i' \in \{1,2,3\}$ and $y \in [-\infty,+\infty]$, we define the further winding number $W_{ii'}(y)$ by the similar formula
\begin{equation}\label{Wii-def}
\begin{split}
W_{ii'}(y) &\coloneqq \frac{1}{2} + \frac{1}{2} \sum_{j=1}^{k_i} \sum_{j'=1}^{k_{i'}} (-1)^{j+j'} \operatorname{sgn}( y_{i,j} + y_{i',j'} - y ) \\
&= \frac{1}{2} \sum_{j=1}^{k_i} \sum_{j'=1}^{k_{i'}} (-1)^{j+j'} (1 + \operatorname{sgn}( y_{i,j} + y_{i',j'} - y )).
\end{split}
\end{equation}
As before, $W_{ii'}(y)$ will be a locally constant integer away from the sums $y_{i,j} + y_{i',j'}$, that equals $0$ for sufficiently large positive $y$ and $1$ for sufficiently large negative $y$.  From Fubini's theorem we have the analogue
\begin{equation}\label{fubini-2}
\int_{-T}^\infty W_{ii'}(y)\ dy = \sum_{j=1}^{k_i} (-1)^{j-1} y_{i,j} + \sum_{j=1}^{k_{i'}} (-1)^{j-1} y_{i',j'} + T
\end{equation}
for sufficiently large $T$.  Curiously, one has the convolution identity
$$ W_{ii'}' = W_i' * W_{i'}'$$
where the primes denote distributional derivatives, although the author was not able to make much use of this identity. The winding numbers $W_{ii'}$ also have some resemblance to the Steinberg formula \cite{steinberg} for the multiplicity of irreducible representations in a tensor product, although this is likely to be just a coincidence.

The hypothesis (ii) of Conjecture \ref{adf} allows us to make the winding number $W_{ii'}$ vanish at some points, and also give some control on the complementary winding number $W_{i''}$:

\begin{proposition}\label{inclusion}  Suppose that the hypotheses of Conjecture \ref{adf} hold.  Let $i,i',i''$ be distinct elements of $\{1,2,3\}$. 
\begin{itemize}
\item[(i)] One has $W_{ii'}(-y_{i'',j}) = 0$ for all $j=0,\dots,k_3+1$.  
\item[(ii)]  If $0 \leq p \leq k_i$ and $0 \leq q \leq k_{i'}$ have the same parity, then one has
\begin{equation}\label{alternate-1}
W_{i''}( - y_{i,p} - y_{i',b} ) = W_{i''}( - y_{i,p+1} - y_{i',b} )
\end{equation}
for $b=q,q+1$ if $|y_{i,p} - y_{i,p+1}| \leq |y_{i',q} - y_{i',q+1}|$, and
\begin{equation}\label{alternate-2}
W_{i''}( - y_{i,a} - y_{i',q} ) = W_{i''}( - y_{i,a} - y_{i',q+1} )
\end{equation}
for $a=p,p+1$ if $|y_{i,p} - y_{i,p+1}| \geq |y_{i',q} - y_{i',q+1}|$.
\end{itemize}
\end{proposition}

\begin{proof}  By permutation we may set $i=1$, $i'=2$, $i''=3$.  Suppose that $0 \leq p_1 \leq k_1$, $0 \leq p_2 \leq k_2$, $0 \leq p_3 \leq k_3$ have the same parity.  By Conjecture \ref{adf}(ii) we have
\begin{align*}
 &\sum_{j_1 = p_1,p_1+1} \sum_{j_2 = p_2,p_2+1} (-1)^{j_1+j_2} \operatorname{sgn}( y_{1,j_1} + y_{2,j_2} + y_{3,p_3} ) \\
 &\quad = \sum_{j_1 = p_1,p_1+1} \sum_{j_2 = p_2,p_2+1} (-1)^{j_1+j_2} \operatorname{sgn}( y_{1,j_1} + y_{2,j_2} + y_{3,p_3+1} );
\end{align*}
summing over $p_1,p_2$ and noting that the contributions of $j_1=0, j_1=k_1+1, j_2=0, j_2=k_2+1$ are the same on both sides we see that
$$ W_{12}(-y_{3,p_3}) = W_{12}(-y_{3,p_3+1})$$
for all $0 \leq p_3 \leq k_3$; since $W_{12}(-y_{3,0}) = W_{12}(+\infty)=+1$, we conclude (i).  

Now suppose $0 \leq p \leq k_1$ and $0 \leq q \leq k_2$ have the same parity and $|y_{1,p} - y_{1,p+1}| \leq |y_{2,q} - y_{2,q+1}|$, and let $0 \leq r \leq k_3$ have the same parity as $p$ and $q$.  From Conjecture \ref{adf}(ii) we see that the pairs $\{ y_{1,p}, y_{1,p+1}\}$, $\{y_{2,q}, y_{2,q+1}\}$, $\{ y_{3,r}, y_{3,r+1}\}$ have non-crossing sums, which by Example \ref{diff} implies that at least one of the pairs $\{ y_{1,p}, y_{1,p+1}\}$, $\{ y_{3,r}, y_{3,r+1}\}$ have no influence on the triple sums.  This implies that
$$
\sum_{j = r,r+1} (-1)^{j-1} \operatorname{sgn}( y_{1,p} + y_{2,b} + y_{3,j} ) 
=  \sum_{j = r,r+1} (-1)^{j-1} \operatorname{sgn}( y_{1,p+1} + y_{2,b} + y_{3,j} )
$$
for $b=q,q+1$; summing in $r$ we obtain \eqref{alternate-1}.  The claim \eqref{alternate-2} is proven similarly.
\end{proof}

This proposition is already enough to reprove the $k_3=1$ case of Conjecture \ref{adf} as follows.  By adding $y_{3,1}$ to all of the $y_{2,j}$ and then sending $y_{3,1}$ to zero, we may assume that $y_{3,1}=0$.  Then we have $W_{13}(y) = W_1(y)$ and $W_{23}(y) = W_2(y)$ for all $y$, and hence by Proposition \ref{inclusion} we have $W_1(-y_{2,j})=0$ for $j=1,\dots,k_2$ and $W_2(-y_{1,j})=0$ for $j=1,\dots,k_1$.  We conclude that on the set $\{ y \in \R: W_1(y) = +1\}$, the function $W_2(-y)$ is locally constant and vanishes at the endpoints, thus we have the inclusion
\begin{equation}\label{peop}
 \{ y \in \R: W_1(y) = +1 \} \subset \{ y \in \R: W_2(-y) = 0 \}.
\end{equation}
This inclusion is strict because the endpoints $y_{1,j_1}$ of the former set cannot match any of the endpoints $-y_{2,j_2}$ of the latter set due to the non-vanishing of the sums $y_{1,j_1} + y_{2,j_2} + y_{3,1} = y_{1,j_1} + y_{2,j_2}$.  We conclude (using Lemma \ref{cjct}) that for sufficiently large $T$, we have
$$ \int_{-T}^\infty W_1(y)\ dy < \int_{-T}^\infty (1-W_2(-y))\ dy$$
and the desired claim \eqref{yio} then follows from \eqref{fubini}, \eqref{fubini-sym}.

\begin{remark}  The above arguments crucially used the hypothesis in Conjecture \ref{adf}(i).  Indeed, the conjecture is false without this hypothesis; a simple counterexample is when $k_1=3$, $k_2=k_3=1$, $y_{1,1} = -1$, $y_{1,2} = -4$, $y_{1,3} = -2$, and $y_{2,1}=y_{3,1}=0$.
\end{remark}

We can partially extend these arguments to cover the cases $k_1,k_2,k_3 > 1$ as follows.  We use Lemma \ref{cjct} to partition
\begin{equation}\label{dos}
\{0,\dots,k_1+1\} \times \{0,\dots,k_2+1\} = V_{12}^0 \cup V_{12}^{+1}
\end{equation}
where $V_{12}^{+1}$ (resp. $V_{12}^0$) consists of those pairs $(p,q)$ for which $W_3(-y_{1,p}-y_{2,q}) = +1$ (resp. $W_3(-y_{1,p}-y_{2,q})=0$). We will work primarily on $V_{12}^{+1}$, although much of the analysis below also applies to $V_{12}^0$.  The set $V_{12}^0$ is a combinatorial analogue of the compact set $K$ in the end of Section \ref{quad}, while $V_{12}^{+1}$ plays the role of the $1$-boundaries $\partial (\pi_3)_* U$ that avoid this compact set.

The set $V_{12}^{+1}$ avoids the boundary of $\{0,\dots,k_1+1\} \times \{0,\dots,k_2+1\}$ and is thus actually a subset of $\{1,\dots,k_1\} \times \{1,\dots,k_2\}$.  We place a directed graph $G_{12}^{+1} = (V_{12}^{+1}, E_{12}^{+1})$ on the vertex set $V_{12}^{+1}$ as follows.  If $0 \leq p \leq k_1$ and $0 \leq q \leq k_2$ have the same parity and $|y_{1,p}-y_{1,p+1}| \leq |y_{2,q}-y_{2,q+1}|$, we connect $(p,b)$ to $(p+1,b)$ whenever $b \in \{q,q+1\}$ is odd with $(p,b) \in V_{12}^{+1}$, and connect $(p+1,b)$ to $(p,b)$ whenever $b \in \{q,q+1\}$ is even with $(p,b) \in V_{12}^{+1}$.  If instead $p,q$ have the same parity and $|y_{1,p}-y_{1,p+1}| > |y_{2,q}-y_{2,q+1}|$, we connect $(a,q)$ to $(a,q+1)$ whenever $a \in \{p,p+1\}$ is odd with $(a,q) \in V_{12}^{+1}$ and connect $(a,q+1)$ to $(a,q)$ whenever $a \in \{p,p+1\}$ is even with $(a,q) \in V_{12}^{+1}$.  By Lemma \ref{inclusion}(ii), this construction only produces edges that start and end in $V_{12}^{+1}$; indeed, every point $(a,b)$ in $\{p,p+1\} \times \{q,q+1\} \cap V_{12}^{+1}$ will be connected to another point in this set, either by an outgoing edge (if $a+b$ has the opposite parity to $p$ or $q$) or an incoming edge (if $a+b$ has the same parity as $p$ or $q$).  Applying this procedure to each square $\{p,p+1\} \times \{q,q+1\}$ with $0 \leq p \leq k_1$ and $0 \leq q \leq k_2$ the same parity, one obtains a directed graph $G_{12}^{+1} = (V_{12}^{+1}, E_{12}^{+1})$ in which each vertex has exactly one outgoing edge and one incoming edge; thus $G_{12}^{+1}$ decomposes into disjoint simple directed cycles.  Any one of these cycles $\gamma$ can enter a vertical line $\{a\} \times \{0,\dots,k_2+1\}$ from the left only when the second coordinate is odd, and from the right only when the second coordinate is even; thus $\gamma$ will intersect such a vertical line at odd second coordinates the same number of times as at even second coordinates; that is to say
$$ \sum_{b: (a,b) \in \gamma} (-1)^b = 0.$$
Similarly for every horizontal line, thus
$$ \sum_{a: (a,b) \in \gamma} (-1)^a = 0$$
for all $0 \leq b \leq k_2+1$.  As a consequence, we have
$$ \sum_{(a,b) \in \gamma} (-1)^{a+b} (y_{1,a} + y_{2,b}) = 0$$
for each cycle $\gamma$, and hence on summing in $\gamma$
$$ \sum_{(a,b) \in V_{12}^{+1}} (-1)^{a+b} (y_{1,a} + y_{2,b}) = 0$$
and hence by \eqref{dos}
\begin{equation}\label{sam}
 \sum_{(a,b) \in V_{12}^0} (-1)^{a+b} (y_{1,a} + y_{2,b}) = \sum_{j=1}^{k_1} (-1)^{j-1} y_{1,j} + \sum_{j=1}^{k_2} (-1)^{j-1} y_{2,j}.
\end{equation}

Next, we claim the identity
\begin{equation}\label{mix}
 \sum_{(a,b) \in V_{12}^{+1}} (-1)^{a+b} \operatorname{sgn}( y_{1,a} + y_{2,b} + y_{3,r} ) = 0
\end{equation}
for all $0 \leq r \leq k_3+1$.  This is certainly the case for $r=0$, so it suffices to show that
$$ \sum_{(a,b) \in V_{12}^{+1}} (-1)^{a+b} \operatorname{sgn}( y_{1,a} + y_{2,b} + y_{3,r} ) = \sum_{(a,b) \in V_{12}^{+1}} (-1)^{a+b} ( \operatorname{sgn}( y_{1,a} + y_{2,b} + y_{3,r+1} ) $$
for all $0 \leq r \leq k_3$.  Fix such a $r$.  By breaking up $V_{12}^{+1}$ into squares, it suffices to show that
\begin{equation}\label{vam}
\begin{split}
& \sum_{(a,b) \in \{ p,p+1\} \times \{q,q+1\} \cap V_{12}^{+1}} (-1)^{a+b} \operatorname{sgn}( y_{1,a} + y_{2,b} + y_{3,r} ) \\
&\quad = \sum_{(a,b) \in \{ p,p+1\} \times \{q,q+1\} \cap V_{12}^{+1}} (-1)^{a+b} \operatorname{sgn}( y_{1,a} + y_{2,b} + y_{3,r+1} ) 
\end{split}
\end{equation}
whenever $0 \leq p \leq k_1$ and $0 \leq q \leq k_2$ have the same parity as $r$.   Suppose first that $|y_{1,p+1}-y_{1,p}| \leq |y_{2,q+1}-y_{2,q}|$, then by Lemma \ref{inclusion}(ii), the set $\{ p,p+1\} \times \{q,q+1\} \cap V_{12}^{+1}$ is the union of horizontal lines $\{(p,b), (p+1,b)\}$.  It then suffices to show that for each such line, we have
\begin{equation}\label{mav}
 \sum_{a \in \{p,p+1\}; c \in \{r',r'+1\}} (-1)^{a+c} \operatorname{sgn}( y_{1,a} + y_{2,b} + y_{3,c}) = 0
\end{equation}
for all $0 \leq r' \leq k_3$ with the same parity as $p,q$; but from Conjecture \ref{adf}(ii) and Example \ref{diff}, the sign of the triple sums of $\{ y_{1,p}, y_{1,p+1}\}$, $\{y_{2,q}, y_{2,q+1}\}$, $\{ y_{3,r'}, y_{3,r'+1}\}$ are not influenced by one of $\{ y_{1,p}, y_{1,p+1}\}$ or $\{ y_{3,r'}, y_{3,r'+1}\}$, giving \eqref{mav}.  The case when $|y_{1,p+1}-y_{1,p}| > |y_{2,q+1}-y_{2,q}|$ is treated similarly (using vertical lines in place of horizontal lines).

If we now define the modified winding number
$$ W_{12}^0(y) := \sum_{(a,b) \in V_{12}^0} (-1)^{a+b} ( \operatorname{sgn}( y_{1,a} + y_{2,b} - y ) $$
then we see from \eqref{mix} and Proposition \ref{inclusion} that
\begin{equation}\label{cob}
 W_{12}^0( - y_{3, r} ) = 0
\end{equation}
for all $0 \leq r \leq k_3+1$.  From \eqref{sam} and Fubini's theorem we see that
$$ \int_{-T}^\infty W_{12}^0(y)\ dy = \sum_{j=1}^{k_1} (-1)^{j-1} y_{1,j} + \sum_{j=1}^{k_2} (-1)^{j-1} y_{2,j} + T$$
and
\begin{equation}\label{breathe}
 \int_{-\infty}^T (1-W_{12}^0(y))\ dy = -\sum_{j=1}^{k_1} (-1)^{j-1} y_{1,j} + \sum_{j=1}^{k_2} (-1)^{j-1} y_{2,j} + T
\end{equation}
for sufficiently large $T$.

On the set $\{ y: W_3(-y) = +1 \}$, we now see that the function $W_{12}^0$ is locally constant (since, by definition of $W_{12}^0$, all the discontinuities $y_{1,p} + y_{2,q}$ of $W_{12}^0$ lie in the set $\{ y: W_3(-y) = 0\}$) and equal to $0$ on the boundary (thanks to \eqref{cob}).  This gives the inclusion
\begin{equation}\label{inclusio}
\{ y: W_3(-y) = +1 \} \subset \{ y: W_{12}^0(y) = 0\}
\end{equation}
which generalises (a permutation of) \eqref{peop}.  This gives some (but not all) cases of Conjecture \ref{adf}:

\begin{proposition}\label{tap}  Conjecture \ref{adf} holds under the additional assumption that the function $W_{12}^0(y) \leq 1$ for all $y$.
\end{proposition}

This case of Conjecture \ref{adf} is analogous to the case of \eqref{dill-3} when the $2$-cycle $\Omega$ appearing in that inequality has a definite sign.

\begin{proof}  The inclusion \eqref{inclusio} is strict, because the endpoints of the set $\{y: W_3(-y) = +1\}$ cannot agree with any of the endpoints of $\{ y: W_{12}^0(y)=0\}$.  We conclude (using Lemma \ref{cjct} and the hypothesis $W_{12}^0 \leq 1$) that for sufficiently large $T$ that
$$ \int_{-\infty}^T W_3(-y)\ dy < \int_{-\infty}^T (1 - W_{12}^0(y))\ dy$$
and the desired inequality \eqref{yio} then follows from \eqref{fubini}, \eqref{breathe}.
\end{proof}

This observation can handle several further cases of Conjecture \ref{adf} (e.g. the perturbative regime in which the $y_{2,1},\dots,y_{2,k_2}$ are very small compared to the differences between the $y_{1,1},\dots,y_{1,k_1}$).  Unfortunately it is possible for $W_{12}^0$ to exceed $1$, which means that one cannot resolve Conjecture \ref{adf} purely on the strength of the inclusion \eqref{inclusio}.  However, it appears from numerous examples that whenever this occurs, a significant portion of the set $\{y: W_{12}(y)^0 = 0\}$ is ``closed off'' from $W_3$, in that the set $\{y: W_3(-y)=+1\}$ is prohibited from entering that portion, which restores the truth of Conjecture \ref{adf}; the author was able to make this statement rigorous in the case $k_3=3$ by a rather lengthy and \emph{ad hoc} argument, which unfortunately does not seem to extend to the general case.  Rather than present this (somewhat unenlightening) argument here, we give an example to illustrate this ``closing off'' phenomenon.  We will take $k_1=k_2=3$ and $y_{1,1} < y_{1,2} < y_{1,3}$ and $y_{2,1} < y_{2,2} < y_{2,3}$ (this ordering is consistent with the non-crossing hypothesis (i)).  We will assume that we are in the ``almost perturbative setting'' in which the nine sums $s_{j_1,j_2} \coloneqq y_{1,j_1}+y_{2,j_2}$ for $j_1,j_2=1,2,3$ are ordered by the relations
$$
s_{1,1} < s_{1,2} < s_{1,3}, s_{2,1} < s_{2,2}  < s_{2,3} <  s_{3,1} < s_{3,2} < s_{3,3}.
$$
thus the only uncertainty in the ordering of these nine sums arises from the relative positions of $s_{1,3}$ and $s_{2,1}$; clearly both orderings are possible.  These relations imply the further inequalities
$$ y_{2,2}-y_{2,1}, y_{2,3}-y_{2,2} < y_{1,2}-y_{1,1}, y_{1,3}-y_{1,2}.$$
By this and many applications of Proposition \ref{inclusion}(ii) we can see that $W_3(-y_{1,p}-y_{2,q})=0$ for all $p,q \in \{0,1,2,3,4\}$, hence $V_{12}^{+1}$ is empty and $W_{12}^0 = W_{12}$ in this case.

First suppose one is in the ``fully perturbative'' setting where $s_{1,3} < s_{2,1}$ (this for instance occurs when all the $y_{2,1}, y_{2,2}, y_{2,3}$ are small compared to the differences $y_{1,2}-y_{1,1}$ and $y_{1,3}-y_{1,2}$).  In this case the winding number $W_{12} = W_{12}^0$ only takes the values $0$ and $1$, with the former occurring on the intervals
\begin{equation}\label{void}
(s_{1,1},s_{1,2}) \cup (s_{1,3}, s_{2,1}) \cup (s_{2,2}, s_{2,3}) \cup (s_{3,1}, s_{3,2}) \cup (s_{3,3}, +\infty),
\end{equation}
and Proposition \ref{tap} gives \eqref{yio} in this case.  In this case one can make the error in \eqref{yio} arbitrarily small; for instance if one takes $k_3=9$ and 
\begin{align*}
y_{3,1} &= - s_{1,1} - \eps \\
y_{3,2} &= - s_{1,2} + \eps \\
y_{3,3} &= - s_{1,3} - \eps \\
y_{3,4} &= - s_{2,3} + \eps \\
y_{3,5} &= - s_{2,2} - \eps \\
y_{3,6} &= - s_{2,1} + \eps \\
y_{3,7} &= - s_{3,1} - \eps \\
y_{3,8} &= - s_{3,2} + \eps \\
y_{3,9} &= - s_{3,3} - \eps 
\end{align*}
for some sufficiently small $\eps > 0$, one can check that the hypotheses of Conjecture \ref{adf} hold, and the left and right-hand sides of \eqref{yio} differ by $9\eps$; see Figure \ref{good}.

\begin{figure} [t]
\centering
\includegraphics[width=5in]{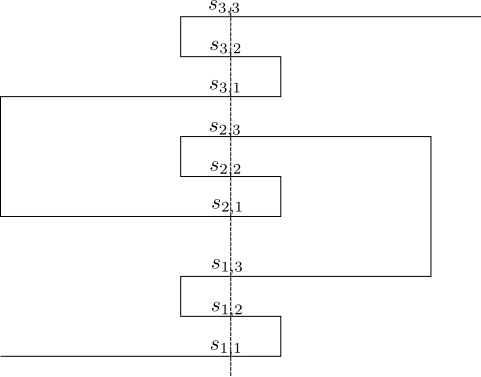}
\caption{The perturbative case.  The solid line represents those sums $(x, y_1+y_2)$, where $(x,y_1)$ lies on a simple curve passing through $(0,y_{1,1}), (0,y_{1,2}), (0,y_{1,3})$, and $(x,y_2)$ lies on a simple curve passing through $(0,y_{2,1}), (0,y_{2,2}), (0, y_{2,3})$.  Note how the entire region \eqref{void} (viewed as a subset of the $y$-axis, drawn here as a dashed line) lies above the solid line, in the sense that it is connected to $(0,T)$ for arbitrarily large and negative $-T$.  The points $(0,-y_{3,i})$ for $i=1,\dots,9$ (not pictured) lie just above curve in this region.}
\label{good}
\end{figure}

Now suppose that $s_{1,3} \geq s_{2,1}$.  In this case, $W_{12} = W_{12}^0$ now takes the value of $0$ on the intervals
$$ (s_{1,1},s_{1,2}) \cup (s_{2,2}, s_{2,3}) \cup (s_{3,1}, s_{3,2}) \cup (s_{3,3}, +\infty)$$
but is additionally equal to $+2$ on the interval $(s_{3,1}, s_{2,3})$.  The argument used to prove Proposition \ref{tap} then fails to establish \eqref{yio}, incurring instead an additional additive error of $s_{1,3}-s_{2,1}$.  However, in this case the portion $(s_{2,2},s_{2,3})$ of $\{ y: W_{12}(y)=1\}$ now becomes ``closed off'' from the points $-y_{3,1},\dots,-y_{3,k_3}$, in the sense that none of the $-y_{3,j}$ can lie in this interval, which also implies that $W_3(-y)$ cannot equal $1$ in this interval either.  This lets one improve the bound arising from \eqref{inclusio} by a factor of $s_{2,3} - s_{2,2}$, which exceeds the loss of $s_{1,3} - s_{2,1}$ incurred previously because $s_{2,1} > s_{1,2}$.  This restores Conjecture \ref{adf} in this case.  To see why none of the $-y_{3,j}$ lie in $(s_{2,2},s_{2,3})$, suppose for contradiction that this were not the case, and let $1 \leq p \leq k_3$ be the largest index such that $-y_{3,p} \in (s_{2,2},s_{2,3})$.  This index $p$ cannot equal $k_3$, because this would imply that the pairs $\{ y_{1,1}, y_{1,2}\}$, $\{ y_{2,3}, y_{2,4} \}$ and $\{ y_{3,p}, y_{3,p+1}\}$ have crossing sums (only one of the eight sums from these pairs is positive), contradicting Conjecture \ref{adf}(ii).  The same argument excludes the case when $p$ is odd and less than $k_3$, since in this case $y_{3,p+1}$ avoids $(s_{2,2},s_{2,3})$ by hypothesis, and also avoids $(s_{1,3},s_{2,2})$ since $W_{12}^0$ equals $+1$ there (here is where we use $s_{1,3} \geq s_{2,1}$), so one has an odd number of positive sums in this case.  Finally, the index $p$ cannot be even, because $-y_{3,p+1}$ lies outside $(s_{2,2}, s_{2,3})$ and also outside $(s_{3,2}, s_{3,3})$ (as $W_{12}$ equals $+1$ there) and hence the triple $\{ y_{1,2}, y_{1,3}\}$, $\{ y_{2,2}, y_{2,3}\}$, $\{y_{3,p}, y_{3,p+1}\}$ would be crossing (this has an odd number of positive sums), again contradicting Conjecture \ref{adf}(ii).  Thus Conjecture \ref{adf} holds in all of these cases.  More generally, the author has observed numerically that every creation of a region where $W_{12}^0$ exceeds $1$ will invariably be accompanied by a larger region of $\{ W_{12}^0(y)=0\}$ which is now ``closed off'' from $W_3$, and was able to verify this claim rigorously when $k_1=3$ or $k_2=3$ by \emph{ad hoc} methods, but was unable to see how to establish such a claim in general.

\begin{figure} [t]
\centering
\includegraphics[width=5in]{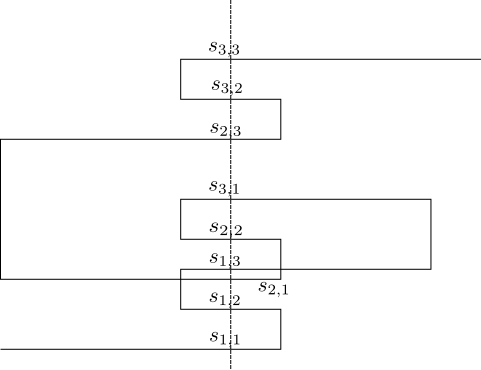}
\caption{The non-perturbative case.  There is now a region of winding number $+2$ between $s_{2,1}$ and $s_{1,3}$.  But to compensate for this, the region between $s_{2,2}$ and $s_{3,1}$, which still has a winding number of $0$, has been cut off from $(0,T)$ for large $T$.  This cut-off region is necessarily larger (as measured as a portion of the $y$-axis) than the region of winding number $+2$.}
\label{sad}
\end{figure}

\end{document}